\documentclass[12pt]{article}
\usepackage{amssymb}
\usepackage{eufrak}
\usepackage{amsmath}
\usepackage{soul}
\usepackage{xr}

\title{On the Definitions of Difference Galois Groups}
\author{Zo{\'e} Chatzidakis\footnote{The author thanks
the Isaac Newton Institute for Mathematical Sciences for its hospitality and financial support
during spring 2005.},  Charlotte Hardouin,  Michael F. Singer\footnote{
The preparation of this paper was supported by NSF Grant CCR-
0096842 and by funds from the Isaac Newton Institute for
Mathematical Sciences during a visit in May 2005.}}

\begin{document}
\def\QED{\hbox{\hskip 1pt \vrule width4pt height 6pt depth 1.5pt \hskip 1pt}}
\newcounter{defcount}[section]
\setlength{\parskip}{1ex}
\newtheorem{thm}{Theorem}[section]
\newtheorem{lem}[thm]{Lemma}
\newtheorem{cor}[thm]{Corollary}
\newtheorem{prop}[thm]{Proposition}
\newtheorem{defin}[thm]{Definition}
\newtheorem{remark}[thm]{Remark}
\newtheorem{remarks}[thm]{Remarks}
\newtheorem{ex}[thm]{Example}
\def\GL{{\rm GL}}
\def\SL{{\rm SL}}
\def\Sp{{\rm Sp}}
\def\sl{{\rm sl}}
\def\sp{{\rm sp}}
\def\gl{{\rm gl}}
\def\PSL{{\rm PSL}}
\def\SO{{\rm SO}}

\def\Gal{{\cal G}al}
\def\Sym{{\rm Sym}}
\def\tildea{\tilde a}
\def\tildeb{\tilde b}
\def\tildeg{\tilde g}
\def\tildee{\tilde e}
\def\tildeA{\tilde A}
\def\calG{{\cal G}}
\def\calH{{\cal H}}
\def\calT{{\cal T}}
\def\calC{{\cal C}}
\def\calP{{\cal P}}
\def\calS{{\cal S}}
\def\calM{{\cal M}}
\def\curve{{\rm {\bf C}}}
\def\P1{{\rm {\bf P}}^1}
\def\Ad{{\rm Ad}}
\def\ad{{\rm ad}}
\def\Aut{{\rm Aut}}
\def\Int{{\rm Int}}
\def\CX{{\mathbb C}}
\def\QX{{\mathbb Q}}
\def\NX{{\mathbb N}}
\def\ZX{{\mathbb Z}}
\def\DX{{\mathbb D}}
\def\AX{{\mathbb A}}
\def\semi{\hbox{${\vrule height 5.2pt depth .3pt}\kern -1.7pt\times $ }}

\newenvironment{prf}[1]{\trivlist
\item[\hskip \labelsep{\bf
#1.\hspace*{.3em}}]}{~\hspace{\fill}~$\square$\endtrivlist}
\newenvironment{proof}{\begin{prf}{Proof}}{\end{prf}}
 \def\square{\QED}
 \newenvironment{proofofthm}{\begin{prf}{Proof of Theorem~\ref{ext}}}{\end{prf}}
 \def\square{\QED}
\newenvironment{sketchproof}{\begin{prf}{Sketch of Proof}}{\end{prf}}

\def\calu{{\cal U}}
\def\si{\sigma}
\def\cee{\CX}
\def\HX{{\mathbb H}}
\def\GX{{\mathbb G}}
\def\Krd{{\rm Kr.dim}}
\def\dcl{{\rm dcl}}

\date{April 24, 2006}
\maketitle
\begin{abstract}{We compare several definitions of the Galois group of a linear difference equation that have arisen in algebra, analysis and model theory and show, that these groups are isomorphic over suitable fields.  In addition, we study properties of Picard-Vessiot extensions over fields with not necessarily algebraically closed subfields of constants.
}
\end{abstract}

\section{Introduction}\label{intro}
In the modern Galois theory of polynomials of degree $n$ with
coefficients in a field $k\footnote{All fields in this paper are
  assumed to be of characteristic zero}$, one associates to a
polynomial $p(x)$ a splitting field $K$,  that is a field $K$ that
is generated over $k$ by the roots of $p(x)$. All such fields are
$k$-isomorphic and  this allows one to define the Galois group of
$p(x)$ to be the group of $k$-automorphisms of such a $K$. If $k$ is
a differential field and $Y' = AY, A$ an $n\times n$ matrix with
entries in $k$, one may be tempted to naively define a ``splitting
field'' for this equation to be a differential field  $K$ containing $k$ and generated (as a differential field) by the entries of a fundamental solution matrix $Z$ of the differential equation\footnote{that is, an invertible $n \times n$ matrix $Z$ such that $Z'=AZ$. Note that the columns of $Z$ form a  basis of the solution space}.  Regrettably, such a field is not unique in general.  For example, for the equation $y' = \frac{1}{2x}y$ over $k = \CX(x), x' = 1$, the fields $k(x^{1/2})$ and $k(z), z$ transcendental over $k$ and $z' = \frac{1}{2x}z$ are not $k$-isomorphic. If one insists that the constants $C_k = \{c \in k \ | \ c'=0\}$ are algebraically closed and that $K$ has no new constants, then Kolchin~\cite{kolchin_exist} showed that such a $K$ exists (and is called the {\em Picard-Vessiot} associated with the equation) and is  unique up to $k$-differential isomorphism. Kolchin~\cite{kolchin48} defined the Galois group of such a field to be the group of $k$-differential automorphisms of $K$ and developed an appropriate Galois theory\footnote{It is interesting to note that the Galois theory was developed before it was known if such $K$ always exist. See the footnote on p.29 of \cite{kolchin48}}. \\[0.1in]
When one turns to difference fields $k$ with automorphism $\sigma$
and difference equations $\sigma Y = AY, \ A\in \GL_n(k)$, the
situation becomes more complicated.
One can consider difference fields $K$ such that $K$ is generated as a difference field by the entries of a fundamental solution matrix.  If  the field of constants $C_k = \{ c \in k \ | \ \sigma(c) = c \}$ is algebraically closed and $K$ has no new constants, then such a $K$ is indeed unique and is again called a Picard-Vessiot extension (\cite{PuSi}, Proposition 1.23 and Proposition 1.9). Unlike the differential case, there are equations for which such a field does not exist. In fact there are difference equations that do not have any nonzero solution in a difference field with algebraically closed constants.  For example, let $K$ be a difference field containing an element $z\neq 0$ such that $\sigma(z) = -z$.  One then has that $z^2$ is a constant.  If, in addition, the constants $C_K$ of $K$ are algebraically closed, then $ z \in C_K$ so $\sigma(z) = z$, a contradiction. This example means that either one must consider  ``splitting fields'' with subfields of constants that are not necessarily algebraically closed or consider ``splitting rings'' that are not necessarily domains. Both paths have been explored and the aim of this paper is to show that they lead, in essence, to the same Galois groups.\\[0.1in]
The field theoretic approach was developed by
Franke\footnote{Bialynicki-Birula \cite{BB2} developed a general
Galois theory for fields with operators but with restrictions that
forced his Galois groups to be connected.} in \cite{franke63} and succeeding
papers.  He showed that for Picard-Vessiot extension fields the
Galois group is a linear algebraic group defined over the constants
and that there is the usual correspondence between closed subgroups
and  intermediate difference fields.
Franke notes that Picard-Vessiot extension fields do not always exist
but does discuss situations when they do exist and results that can be
used when adjoining solutions of a linear difference equation forces one
to have new constants.\\[0.1in] 
Another field theoretic approach is contained in the work of Chatzidakis
and Hrushovski \cite{ChHr1}.  Starting from a difference field $k$, they
form a certain large difference extension  $\calu$ having the properties
(among others) that  for any element in $\calu$ but not in $k$, there is
an automorphism of $\calu$ that moves this element and that any set of
difference  equations (not necessarily  linear)  that have a solution in
some extension of $\calu$ already have a solution in $\calu$. The
subfield of constants $C_\calu$ is not an algebraically closed field.
Given a linear difference equation with coefficients in $k$, there
exists a fundamental solution matrix with entries in $\calu$.  Adjoining
the entries of these to $k(C_\calu)$ yields a difference field $K$. A
natural candidate for a Galois group  is the  group of difference
automorphisms of $K$ over $k(C_\calu)$ and these do indeed correspond to
points in a linear algebraic group. Equality of this automorphism group
with the Galois group coming from 
Picard-Vessiot rings is shown in \ref{zp5} under certain 
conditions (which are always verified when $C_k$ is algebraically
closed). Proofs are very algebraic in nature, and along the way produce
some new algebraic results on Picard-Vessiot rings: we find numerical
invariants of Picard-Vessiot rings of the equation $\sigma(X)=AX$, and
show how to compute them (see \ref{zp3} and \ref{zp4}). Furthermore, we show how to
compute the number of primitive idempotents of a Picard-Vessiot ring
when the field $C_k$ is algebraically closed (\ref{zcor3}).  This situation will be further
discussed in Section~\ref{mt}.\\[0.1in] 
The field theoretic approach also seems most natural in the analytic situation. For example, let $\calM(\CX)$ be the field of functions $f(x)$ meromorphic on the complex plain endowed with the automorphism defined by the shift $\sigma(x) = x+1$. Note that the constants $C_{\calM(\CX)}$ are the periodic meromorphic functions.  A theorem of Praagman \cite{praagman} states that a difference equation with  coefficients in $\calM(\CX)$ will have a fundamental solution matrix  with entries in $\calM(\CX)$. If $k$ is the smallest difference field containing the coefficients of the equation and $C_{\calM(\CX)}$ and $K$ is the smallest difference field containing $k$ and the entries of fundamental solution matrix, then, in this context, the natural Galois group is the set of difference automorphisms of $K$ over $k$. For example, the difference equation $\sigma(y) = -y$ has the solution $y = e^{\pi i x}$. This function is algebraic of degree $2$ over the periodic functions $k =  C_{\calM(\CX)}$. Therefore, in this context the Galois group of $K = k(e^{\pi i x})$ over $k$ is $\ZX/2\ZX$.\\[0.1in]
  One can also consider the field $\calM(\CX^*)$ of meromorphic functions on the punctured plane $\CX^* = \CX \backslash \{0\}$ with $q-$automorphism $\sigma_q(x) = qx, \  |q| \neq 1$. Difference equations in this context are $q$-difference equations and Praagman proved a global existence theorem in this context as well. The constants $C_{\calM(\CX^*)}$ naturally correspond to meromorphic functions on the elliptic curve $\CX^*/q^{\ZX}$ and one can proceed as in the case of the shift.   One can also define local versions (at infinity in the case of the shift and at zero or infinity in the case of $q$-difference equations). In the local case and for certain restricted equations  one does not necessarily need constants beyond those in $\CX$ (see \cite{etingof},  \cite{vdp_reversat}, \cite{PuSi} as well as connections between the local and global cases. Another approach to $q$-difference equations is given by Sauloy in \cite{sauloy_galois} and Ramis and Sauloy in \cite{ramis_sauloy} where a Galois group is produced using a combination of analytic and tannakian tools.  The Galois groups discussed in these papers do not appear to act on  rings or fields and, at present, it is not apparent  how the techniques presented here can be used to compare these groups to  other putative Galois groups.)\\[0.1in]
An approach to the Galois theory of difference equations with coefficients in difference fields based on rings that are not necessarily integral was presented in \cite{PuSi} (and generalized by Andr{\'e} in \cite{andre_galois} to include differential and difference equations with coefficients in fairly general rings as well). One defines a {\em Picard-Vessiot ring} associated with a difference equation $\sigma Y = AY$ with coefficients in a difference field $k$ to be a  simple difference ring (i.e., no $\sigma$-invariant ideals) $R$ of the form $R = k[z_{i,j}, 1/\det(Z)]$ where $Z = (z_{i,j})$ is a fundamental solution matrix  of $\sigma Y=AY$.  Assuming that $C_k$ is algebraically closed, it is shown in \cite{PuSi} that such a ring {\em always} exists and is unique up to $k$-difference isomorphism. A similar definition for differential equations yields a ring that is an integral domain and leads (by taking the field of quotients) to the usual theory of Picard-Vessiot extensions (see \cite{PuSi2003}).  In the difference case, Picard-Vessiot rings need not be domains.  For example, for the field $k = \CX$ with the trivial automorphism, the Picard-Vessiot ring corresponding to $\sigma y = -y$ is $\CX[Y]/(Y^2-1), \sigma(Y) = -Y$.   Nonetheless, one defines the {\em difference Galois group of $\sigma Y= AY$} to be the $k$-difference automorphisms of $R$ and one can shows that this is a linear algebraic group defined over $C_k$. In the example above, the Galois group is easily seen to be $\ZX/2\ZX$.  Furthermore,  in general there is a Galois correspondence between certain subrings of the total quotient ring and closed subgroups of the Galois group.\\[0.1in]
The natural question arises: {\em How do these various groups relate
to each other?} The example of $\sigma(y) = -y$ suggests that the
groups may be the same. Our main result, Theorem~\ref{thethm},
states that all these groups are isomorphic as algebraic groups over
a suitable extension of the constants. This result  has interesting
ramifications for the analytic theory of difference equations. In
\cite{hardouin_thesis}, the second author gave criteria to insure
that solutions, meromorphic in $\CX^*$,  of a first order
$q$-difference equation over $\CX(x)$ satisfy no  algebraic differential relation over $C_{\calM(\CX^*)}(x)$, where $C_{\calM(\CX^*)}$  is the field of meromorphic functions on the elliptic curve $\CX^*/q^{\ZX}$. The proof of this result presented in \cite{hardouin_thesis}  depended on knowing the dimension of  Galois groups in the analytic (i.e., field-theoretic) setting. These groups could be calculated in the ring theoretic setting of \cite{PuSi} and the results of the present paper allow one to transfer this information to the analytic setting.  Although we will not go into more detail concerning the results of \cite{hardouin_thesis}, we will give an example of how one can deduce transcendence results in the analytic setting from their counterparts in the formal setting.\\[0.1in]
The rest of the paper is organized as follows. In Section~\ref{galoissec}, we show how results of \cite{PuSi} and \cite{PuSi2003} can be modified to prove the correspondence of various Galois groups.  In Section~\ref{tannaka} we prove this result again in the special case of $q$-difference equations over $\CX(x)$ using tannakian tools in the spirit of Proposition 1.3.2 of \cite{katz_calculations}. In Section~\ref{mt}, we discuss the model-theoretic approach in more detail and, from this point of view, show the correspondence of the Galois groups. In addition, we consider some additional properties of Picard-Vessiot rings over fields with constant subfields that are not necessarily algebraically closed. The different approaches and proofs have points of contacts (in particular, Proposition~\ref{p2}) and we hope comparisons of these techniques are enlightening. \\[0.1in]
The authors would like to thank Daniel Bertrand for suggesting the
approach of Section~\ref{tannaka} and his many other useful comments
concerning this paper.

\section{A Ring-Theoretic Point of View}\label{galoissec} In this section we shall consider groups of difference automorphisms of rings   and fields generated by solutions of linear difference equations and show that these groups are isomorphic, over the algebraic closure of the constants to the Galois groups defined in \cite{PuSi2003}. We begin by defining the rings and fields we will study.

\begin{defin} Let $K$ be a difference field with automorphism $\sigma$ and let $A \in \GL_n(K)$.\\[0.1in]
a.~We say that a difference ring extension $R$ of $K$ is a {\em weak
Picard-Vessiot ring} for  the equation $\sigma X = AX$ if
\begin{itemize}
\item[(i)]  $R = K[Z,\frac{1}{\det(Z)}]$ where $Z \in \GL_n(R)$ and $\sigma Z = AZ$ and
 \item[(ii)]  $C_R = C_K$
 \end{itemize}
 \noindent b.~We say that a difference field extension $L$ of $K$ is a {\em weak Picard-Vessiot field} for $\sigma X = AX$ if $C_L = C_K$ and $L$ is the quotient field of a weak Picard-Vessiot ring of $\sigma X = AX$.
 \end{defin}
In \cite{PuSi}, the authors define a {\em Picard-Vessiot ring}  for the equation $\sigma Y = AY$  to be a difference ring $R$ such that $(i)$ holds and in addition $R$ is simple as a difference ring, that is, there are no $\sigma$-invariant ideals except $(0)$ and $R$.  When $C_K$ is algebraically closed, Picard-Vessiot rings exist, are unique up to $K$-difference isomorphisms  and have the same constants as $K$ (\cite{PuSi}, Section 1.1).  Therefore in this case, the Picard-Vessiot ring will be a weak Picard-Vessiot ring.\\[0.1in]
In general, even when the field of constants is algebraically closed, Example 1.25 of \cite{PuSi} shows that there will be weak Picard-Vessiot rings that are not Picard-Vessiot rings. Furthermore this example shows that the quotient field of a weak Picard-Vessiot integral domain $R$ need not necessarily have the same constants as $R$ so the requirement that $C_L = C_K$ is not superfluous.\\[0.1in]
The Galois theory of Picard-Vessiot rings is developed  in
\cite{PuSi} for Picard-Vessiot rings $R$ over difference fields $K$
with algebraically closed constants $C_K$.  In particular, it is
shown (\cite{PuSi}, Theorem 1.13) that the groups of difference
$K$-automorphisms of $R$ over $K$ corresponds to the set of
$C_K$-points of a linear algebraic group defined over $C_K$.  A
similar result for differential equations is proven in
(\cite{PuSi2003}, Theorem 1.27).  It has been observed by many
authors beginning with Kolchin (\cite{DAAG}, Ch.~VI.3 and VI.6;
others include \cite{andre_galois}, \cite{deligne_milne},
\cite{deligne_tannakian}, \cite{katz_calculations},
\cite{papanikolas} in a certain characteristic $p$ setting for
difference equations) that one does not need $C_k$ to be
algebraically closed to achieve this latter result.  Recently,
Dyckerhoff \cite{dyckerhoff} showed how the proof of Theorem 1.27 of
\cite{PuSi2003} can be adapted in the differential case to fields
with constants that are not necessarily algebraically closed. We
shall give a similar adaption in the difference case.

\begin{prop}\label{p1} Let $K$ be a difference field of characteristic zero and let $\sigma Y = AY, A\in \GL_n(K)$ be a difference equation over $K$. Let $R$ be a weak Picard-Vessiot ring for this equation over $K$. The group  of difference $K$-automorphisms of $R$ can be identified with the $C_K$-points of a linear algebraic group $G_R$ defined over $C_K$.
 \end{prop}
 \begin{proof} We will define the group $G_R$ by producing a representable functor from the category of commutative $C_K$-algebras to the category of groups (c.f., \cite{waterhouse}).\\[0.1in]
 First, we may write $R = K[Y_{i,j},\frac{1}{\det(Y)}]/ q$ as the quotient of a difference ring $K[Y_{i,j},\frac{1}{\det(Y)}]$, where $Y=\{Y_{i,j}\}$ is an $n\times n$ matrix of indeterminates with $\sigma Y = AY$, by a $\sigma$-ideal $q$.  Let $C= C_K$. For any $C$-algebra $B$, one defines the difference rings $K\otimes_{C}B$ and $R\otimes_{C} B$ with automorphism $\sigma(f\otimes b) = \sigma(f)\otimes b$ for $f \in K$ or $R$.  In both cases, the ring of constants is $B$. We define the functor $\calG_{R}$ as follows: the group $\calG_{R}(B)$ is the group of $K\otimes_{C} B$-linear automorphisms of $R\otimes_{C} B$ that commute with $\sigma$. One can show that $\calG_{R}(B)$ can be identified with the group of matrices $M \in \GL_n(B)$ such that the difference automorphism $\phi_M$ of $R\otimes_{C} B$, given by $(\phi_M Y_{i,j}) = (Y_{i,j})M$, has the property that $\phi_M(q) \subset (q)$ where $(q)$ is the ideal of $K[Y_{i,j}, \frac{1}{\det(Y_{i,j})}]\otimes_{C}B$ generated by $q$. \\[0.1in]
We will now show that $\calG_{R}$ is representable.  Let $X_{s,t}$
be new indeterminates and  let $M_0 = (X_{s,t})$. Let $q = (q_1,
\ldots , q_r)$ and write $\sigma_{M_0}(q_i)\mod(q) \in
{R}\otimes_{C} {C}[X_{s,t}, \frac{1}{\det(X_{s,t})}]$ as a finite
sum
\[ \sum_i C(M_0,i,j)e_i\mbox{ with all } C(M_0,i,j) \in {C}[X_{s,t}, \frac{1}{\det(X_{s,t})}] \ ,\]
where $\{e_i\}_{i \in I}$ is a ${C}$-basis of $R$.  Let $I$ be the the
ideal in ${C}[X_{s,t}, \frac{1}{\det(X_{s,t})}]$ generated by all the
$C(M_0, i,j)$.  We will show that 
\[U:= {C}[X_{s,t}, \frac{1}{\det(X_{s,t})}]/I\] 
represents $\calG_{R}$.\\[0.1in]
Let $B$ be a ${C}$-algebra and $\phi \in \calG_{R}(B)$ identified
with $\phi_M$ for some $M \in \GL_n(B)$. One defines the
${C}$-algebra homomorphism 
\[\Phi: {C}[X_{s,t},
\frac{1}{\det(X_{s,t})}] \rightarrow B,\qquad (X_{s,t}) \mapsto M.\]
 The
condition on $M$ implies that the kernel of $\Phi$ contains $I$.
This then gives a unique ${C}$-algebra homomorphism $\Psi:U
\rightarrow B$ with $\Psi(M_0\mod I) = M$. The Yoneda Lemma can now
be used to show that $G_R =$ Spec($U$) is a linear algebraic group
(see Appendix B, p.~382 of \cite{PuSi2003} to see how this is
accomplished or Section 1.4 of \cite{waterhouse}).\end{proof}
We will refer to $G_R$ as the {\em Galois group} of $R$. When $R$ is
a Picard-Vessiot extension of $K$, we have the usual situation.
We are going to compare the groups associated with  a Picard-Vessiot
extension and weak Picard-Vessiot field extensions for the same
equation over different base fields.  We will first show that
extending a Picard-Vessiot ring by constants yields a Picard-Vessiot
ring whose associated group  is isomorphic to the original group
over the new constants.  In the differential case and when the new
constants are algebraic over the original constants this appears in
Dyckerhoff's work (\cite{dyckerhoff}, Proposition 1.18 and Theorem
1.26).   Our proof is in the same spirit but without appealing to
descent techniques. We will use Lemma 1.11 of \cite{PuSi}, which we
state here  for the convenience of the reader:
\noindent \begin{lem} \label{lem1.11} Let $R$ be a Picard-Vessiot
ring over a field $k$ with $C_R = C_k$\footnote{The hypothesis
$C_R=C_k$ is not explicitly stated in the statement of this result
in \cite{PuSi} but is assumed in the proof.} and $A$ be a
commutative algebra over $C_k$.  The action of $\sigma$ on $A$ is
supposed to be the identity.  Let $N$ be an ideal of $R\otimes_{C_k}
A$ that is invariant under $\sigma$. Then $N$ is generated by the
ideal $N\cap A$ of $A$. \end{lem}
 \begin{prop}\label{p2} Let $k \subset K$ be  difference fields of characteristic zero  and $K = k(C_K)$. Let $R$ be a Picard-Vessiot ring over $k$ with $C_{R} = C_k$ for the equation $\sigma X = AX, A\in \GL_n(k)$.   If
$R = k[Y,\frac{1}{\det(Y)}]/q$ where $Y$ is an $n\times n$ matrix of
indeterminates, $\sigma Y = AY$ and $q$ is a maximal $\sigma$-ideal,
then $S = K[Y,\frac{1}{\det(Y)}]/qK$ is a Picard-Vessiot extension
of  $K$ for the same equation.
Furthermore, $C_S = C_K$.\end{prop}
 \begin{proof} First note that the ideal $qK\neq K[Y,\frac{1}{\det(Y)}].$  Secondly, Lemma~\ref{lem1.11} states that for $R$ as above and $A$ a commutative $C_k$ algebra with identity, any $\sigma$-ideal $N$ of $R\otimes_{C_k}A$ (where the action of $\sigma$ on $A$ is trivial) is generated by $N\cap A$. This implies that the difference ring $R\otimes_{C_k}C_K$ is simple. Therefore the map $\psi: R\otimes_{C_k}C_K \rightarrow S = K[Y,\frac{1}{\det(Y)}]/(q)K$  where $\psi(a\otimes b) = ab$ is injective. Let $R'$ be the image of $\psi$. One sees that any element of $S$ is of the form $\frac{a}{b}$ for some $a \in R', b \in k[C_k]\subset R'$.  Therefore any ideal $I$ in $S$ is generated by $I\cap R'$ and so $S$ is simple. \\[0.1in]
For any constant  $c\in S$, the  set $J = \{a \in R' \ | \ ac\in
R'\} \subset R'$  is a nonzero $\sigma$-ideal so  $c \in R'$. Since
the constants of $R'$ are $C_K$, this completes the
proof.\end{proof}
%
\begin{cor} \label{cor1} Let $R$ and $S$ be as in Proposition~\ref{p2}.
If $G_{R}$ and $G_S$ are the Galois groups associated with these
rings as in Proposition~\ref{p1}, then $G_{R}$ and $G_S$ are
isomorphic over $C_K$.\end{cor} 
\begin{proof} We are considering $G_{R}$ as the functor from $C_k$
  algebras $A$ to groups defined by $G_{R}(A) :=  Aut(R\otimes_{C_k}A)$ where $Aut(..)$ is the group of difference  $k\otimes A$-automorphisms. Let $T_{R}$ be the finitely generated $C_k$-algebra representing $G_{R}$ (i.e., the coordinate ring of the group).  Similarly, let $T_S$ be the $C_K$-algebra representing $G_S$. We define a new functor $F$ from $C_K$-algebras to groups as $F(B):= Aut((R\otimes_{C_k}C_K) \otimes_{C_K} B)$. One checks that $F$ is also a representable functor represented by $T_{R}\otimes_{C_k}C_K$. Using the embedding $\psi$ of the previous proof, one sees that $F(B) = Aut(S\otimes_{C_K}B) = G_R(B)$ for any $C_K$-algebra $B$. The Yoneda Lemma implies that $T_{R}\otimes_{C_k} C_K \simeq T_S$.\end{proof}
In Proposition~\ref{p3} we will compare Picard-Vessiot rings with
weak Picard-Vessiot fields for the same difference equation.  To do
this we need the following lemma.  A version of this in the
differential case appears as Lemma 1.23 in \cite{PuSi2003}.
\begin{lem}\label{lem1} Let $L$ be a difference field. Let $Y=(Y_{i,j})$ be  and $n\times n$ matrix of indeterminates and extend $\sigma$ to $L[Y_{i.j}, \frac{1}{\det(Y)}]$ by setting $\sigma(Y_{i,j}) = Y_{i,j}$. The map $I \mapsto (I) = I\cdot  L[Y_{i.j}, \frac{1}{\det(Y)}]$  from the set of ideals in $C_L[Y_{i.j}, \frac{1}{\det(Y)}]$ to the set of ideals of $L[Y_{i.j}, \frac{1}{\det(Y)}]$ is a bijection.\end{lem}
\begin{proof} One easily checks that $(I) \cap C_L[Y_{i.j}, \frac{1}{\det(Y)}] = I$. Now, let $J$ be an ideal of $L[Y_{i.j}, \frac{1}{\det(Y)}]$ and let $I = J \cap C_L[Y_{i.j}, \frac{1}{\det(Y)}] $. Let $\{e_i\}$ be a basis of $C_L[Y_{i.j}, \frac{1}{\det(Y)}]$ over $C_L$. Given $f \in L[Y_{i.j}, \frac{1}{\det(Y)}]$, we may write $f$ uniquely as $f = \sum f_ie_i, \ f_i \in L$. Let $\ell(f)$ be the number of $i$ such that $f_i \neq 0$.  We will show, by induction on $\ell(f)$, that for any $f \in J$, we have $f\in (I)$. If $\ell(f) = 0, 1$ this is trivial.  Assume $\ell(f) >1$.  Since $L$ is a field, we can assume that there exists an $i_1$ such that $f_{i_1} = 1$. Furthermore, we may assume that there is an $i_2 \neq i_1$ such that $f_{i_2} \in L\backslash C_L$.  We have $\ell(f - \sigma(f)) < \ell(f)$ so $\sigma(f) - f \in (I)$.  Similarly, $\sigma(f_{i_2}^{-1} f) - f_{i_2}^{-1} f \in (I)$.  Therefore, $(\sigma(f_{i_2}^{-1}) - f_{i_2}^{-1}) f = \sigma(f_{i_2}^{-1})(f-\sigma(f)) + (\sigma(f_{i_2}^{-1} f) - f_{i_2}^{-1} f) \in (I)$. This implies that $f \in (I)$. \end{proof}
The following is a version of Proposition 1.22 of \cite{PuSi2003}
modified for difference fields  taking into account the possibility
that the constants are not algebraically closed.

\begin{prop} \label{p3} Let $K$ be a difference field with constants $C$ and let $A \in \GL_n(K)$. Let $S = K[U,\frac{1}{\det(U)}],  \ U \in \GL_n(S), \ \sigma(U) = AU$ be a Picard-Vessiot extension of $K$ with $C_S = C_k$ and  let  $L = K(V), \ V \in GL_n(L), \ \sigma(V) = AV$ be a weak Picard-Vessiot field extension of $K$. Then there exists a $K$-difference embedding $\rho:S \rightarrow L\otimes_C \overline{C}$ where $\overline{C}$ is the algebraic closure of $C$ and $\sigma$ acts on $L\otimes_C \overline{C}$ as $\sigma(v\otimes c) = \sigma(v)\otimes c$.\end{prop}

\begin{proof} Let $X = (X_{i,j})$ be an $n\times n$ matrix of indeterminates over $L$ and let $S_0:= K[X_{i,j}, \frac{1}{\det(X)}] \subset L[X_{i,j}, \frac{1}{\det(X)}]$. We define a difference ring structure on $L[X_{i,j}, \frac{1}{\det(X)}]$ by setting $\sigma(X) = AX$ and this gives a difference ring structure on $S_0$. Abusing notation slightly,  we may write $S = S_0/p$ where $p$ is a maximal $\sigma$-ideal of $S_0$.   Define elements $Y_{i,j} \in L[X_{i,j}, \frac{1}{\det(X)}]$ via the formula  $(Y_{i,j}) = V^{-1}(X_{i,j})$.  Note that $\sigma Y_{i,j} = Y_{i,.j}$ for all $i,j$ and that  $L[X_{i,j}, \frac{1}{\det(X)}] = L[Y_{i,j}, \frac{1}{\det(Y)}]$.  Define $S_1 := C[Y_{i,j}, \frac{1}{\det(Y)}]$.  The ideal $p \subset S_0 \subset L[Y_{i,j}, \frac{1}{\det(Y)}]$ generates an ideal $(p)$ in $L[Y_{i,j}, \frac{1}{\det(Y)}]$.  We define $\tilde{p} = (p) \cap S_1$. Let $m$ be a maximal ideal in $S_1$ such that $\tilde{p} \subset m$.  We then have a homomorphism $S_1 \rightarrow S_1/m \rightarrow \overline{C}$. We can extend this to a homomorphism $\psi: L[Y_{i,j}, \frac{1}{\det(Y)}]=L\otimes_CS_1 \rightarrow L\otimes_C\overline{C}$.  Restricting $\psi$ to $S_0$, we have a difference homomorphism $\psi: S_0 \rightarrow L\otimes_C\overline{C}$ whose kernel contains $p$. Since $p$ is a maximal $\sigma$-ideal we have that this kernel is $p$.  Therefore  $\psi$ yields an embedding  $\rho:S = S_0/p\rightarrow L\otimes_C\overline{C}$.\end{proof}

\begin{cor}\label{cor2} Let $K, C, \overline{C}, S, L$ and $\rho$ be as above and let $T = K[V,\frac{1}{\det(V)}]$. Then $\rho$ maps $S\otimes_C\overline{C}$ isomorphically onto $T\otimes_C\overline{C}$. Therefore the Galois group $G_S$ is isomorphic to $G_T$ over $\overline{C}$.\end{cor}
\begin{proof} In Proposition~\ref{p3}, we have that $\rho(U) = V(c_{i,j})$ for some $(c_{i,j}) \in \GL_n(\overline{C})$. Therefore $\rho$ is an isomorphism. The isomorphism of $G_S$ and $G_T$ over $\overline{C}$ now follows in the same manner as the conclusion of Corollary~\ref{cor1}.\end{proof}
We can now prove the following result.
\begin{thm}\label{thethm} Let \begin{enumerate}
 \item $k$ be a difference field with algebraically closed field of constants $C$,
 \item  $\sigma Y = AY$ be a difference equation with $A\in  \GL_n(k)$ and let $R$ be the Picard-Vessiot ring for this equation over $k$,
 \item $K$ a difference field extension of $k$ such that $K = k(C_K)$
 \item $L$ a weak Picard-Vessiot field for the equation $\sigma(Y) = AY$ over $K$.
 \end{enumerate}
 Then \begin{enumerate}
 \item[a.] If we write $L = K(V)$ where $ V \in \GL_n(L) \mbox{ and }\sigma V = AV$ then $R\otimes _{C}\overline{C}_K \simeq K[V,\frac{1}{\det(V)}]\otimes_{C_K} \overline{C}_K $ where $\overline{C}_K $ is the algebraic closure of $C_K$. Therefore $K[V,\frac{1}{\det(V)}]$ is also a Picard-Vessiot extension of $K$.
 \item[b.] The Galois groups of $R$ and $K[V,\frac{1}{\det(V)}]$  are isomorphic over $\overline{C}_K$.
 \end{enumerate}
 \end{thm}
 \begin{proof} Let $Y= (Y_{i,j})$ be an $n\times n$ matrix of indeterminates and write $R = k[Y_{i,j}, \frac{1}{\det(Y)}]/(p)$, where $(p)$ is a maximal $\sigma$-ideal.  Assumptions $1.$ and $2.$ imply that $C_{R} = C_k$ (\cite{PuSi},Lemma 1.8) so Propostion~\ref{p2} implies that $S = K[Y_{i,j}, \frac{1}{\det(Y)}]/(p)K$ is a Picard-Vessiot ring with constants $C_K$. Corollary~\ref{cor1} implies that its Galois group $G_R$ is isomorphic over $C$ to $G_{S}$.  Corolary~\ref{cor2} finishes the proof.
 \end{proof}

%

\section{A Tannakian Point of View}\label{tannaka}

In this section we shall give another proof of Theorem~\ref{thethm} for $q$-difference equations in the analytic situation.   Let $\calM(\CX^*)$ be the field of functions $f(x)$ meromorphic on $\CX^* = \CX\backslash \{0\}$ with the automorphism $\sigma(f(x)) = f(qx)$ where $q \in \CX^*$ is a fixed complex number with $|q| \neq 1$. As noted before, the constants $C_{\calM(\CX^*)}$ in this situation correspond to meromorphic functions $\calM(E)$ on the elliptic curve $E=\CX^*/q^\ZX$.  We shall show how the theory of tannakian categories also yields a proof of Theorem~\ref{thethm} when $k = \CX(x)$ and $K = k(C_{\calM(\CX^*)})$.\\[0.1in]
We shall assume that the reader is familiar with some basic facts
concerning difference modules (\cite{PuSi}, Ch. 1.4)  and tannakian
categories  (\cite{deligne_milne},\cite{deligne_tannakian}; see
\cite{PuSi2003}, Appendix B or \cite{breen} for an overview).   We
will denote by $\mathcal{D}_k= k[\sigma, \sigma^{-1}]$ (resp.
$\mathcal{D}_{K}=K[\sigma, \sigma^{-1}]$) the rings of difference
operators over $k$ (resp. $K$). Following (\cite{PuSi}, Ch. 1.4), we
will denote by $Diff(k, \sigma)$ (resp. by $Diff(K, \sigma)$) the
category of difference-modules over $k$ (resp. $K$). The ring of
endomorphisms of the unit object is equal to $\CX$
(resp. $C_K = C_{\calM(\CX^*)} = \calM(E)$) the field of constants of $k$ (resp. $K$).\\[0.1in]
Let  $M$ be a  $\mathcal{D}_k$-module of finite type over  $k$. We
will denote by $M_K =M\otimes_k K$ the $\mathcal{D}_K$-module
constructed by extending the field $k$ to $K$.  We will let
$\lbrace \lbrace M \rbrace \rbrace$ (resp.  $\lbrace \lbrace M_K
\rbrace \rbrace$) denote the full abelian tensor subcategory of
$Diff(k, \sigma)$ (resp.  $Diff(K, \sigma)$)
generated by $M$ (resp. $M_K$) and its dual $M^*$ (resp. ${M_K}^*$).\\[0.1in]
Theorem $1.32$ of \cite{PuSi} gives  a fiber functor $\omega_M$
 over $\mathbb{C}$
 for $\lbrace \lbrace M \rbrace
\rbrace$. In \cite{praagman}, Praagman  gave an existence theorem
(see Section~\ref{intro}) for $q$-difference equations which can be
used to construct a fiber functor $\omega_{M_K}$ for
 $\lbrace \lbrace M_K \rbrace \rbrace$ over $C_K$ (described in detail in Proposition \ref{prag} below). In particular, $\lbrace \lbrace M \rbrace
\rbrace$ and  $\lbrace \lbrace M_K \rbrace \rbrace$ are neutral
tannakian categories over $\mathbb{C}$ and $C_K$ respectively.  The
main task of
 this section is to compare the Galois groups associated to the
 fiber functors $\omega_M$ and $\omega_{M_K}$.  We will prove the
 following theorem:
\begin{thm}\label{ext}
Let $M \in Diff(k, \sigma)$ be a $\mathcal{D}_k$-module of finite
type over $k$.

Then
$$ Aut^{\otimes}( \omega_M) \otimes_{\mathbb{C}} \overline{C_K} \simeq  Aut^{\otimes}(\omega_{M_K}) \otimes_{C_K} \overline{C_K}.$$
\end{thm}
The proof is divided in two parts. In the first part, we will
construct a fiber functor $\tilde{\omega}_M$  from  $\lbrace \lbrace
M_K \rbrace \rbrace$ to $Vect_{C_K}$, which \textit{extends}
$\omega_M$ and we will compare its Galois group to that associated
to $\omega_M$. In the second part, we will compare the Galois group
associated to $\omega_{M_K}$ and the Galois group associated to
$\tilde{\omega}_M$, and finally relate these groups to the Galois
groups considered in Theorem~\ref{thethm}.b).

\subsection{The action of $\Aut(C_K/\CX)$ on  $\lbrace \lbrace M_K \rbrace
\rbrace$}\label{autsec}

A module $M_K = M\otimes_k K$ is constructed from the module $M$
essentially by extending the scalars from $\CX$ to $C_K$.  In order
to compare the subcategories $\{\{M\}\}$ and $\{\{M_K\}\}$ they
generate,  it seems natural therefore to consider an action of the
automorphism group $\Aut(C_K/\CX)$ on $M_K$ as well as on
$\{\{M_K\}\}$.  Before we define this action we state some
preliminary facts.

\begin{lem}\label{galois} We have :
\begin{enumerate}
\item  The fixed field $C_K^{\Aut(C_K/\CX)}$ is $\mathbb{C}$.
\item $K \simeq C_K(X)$ where $C_K(X)$ denotes the field of
  rational functions with coefficients in $C_K$. This isomorphism maps
   $\mathbb{C}(X)$  isomorphically onto $k$.
\end{enumerate}
\end{lem}
\begin{proof} $1.$ For all $c \in \mathbb{C}^*$, the restriction to $C_K$ of the map $
  \sigma_c$ which associates to $f(x)\in C_K$ the function $
  \sigma_c(f)(x)=f(cx )$ defines an element of $\Aut(C_K/\CX)$. Let $\phi \in
  C_K^{\Aut(C_K/\CX)}$, the fixed field of $\Aut(C_K/\CX)$. Because $\sigma_c(\phi)=\phi$ for any  $ c\in \mathbb{C}^*$,
  $\phi$ must be  constant.\\[0.1in]
$2.$ For any  $f(X) \in C_K[X]$, put $\phi(f) = f(z)$, viewed as a
meromorphic function of the variable  $z \in  \mathbb{C}^*$. Then,
$\phi$ is a  morphism from $C_E[X]$ to
  $K_E$. We claim that $\phi$ is injective. Indeed,  let us consider a
  dependence relation :
\begin{equation} \label{eqn:2} \sum c_i(z)k_i(z)=0, \forall z \in \mathbb{C}
\end{equation}
where $c_i \in C_E$ and $k_i \in K$. Using Lemma II of (\cite{cohn},
p.~271) or the Lemma of (\cite{etingof}, p.~5) the relation
(\ref{eqn:2}) implies that
\begin{equation} \label{eqn:3} \sum c_i(z)k_i(X)=0, \forall z \in
  \mathbb{C}. \\
\end{equation}
So $\phi$ extends to the function field $C_K(X)$, whose image is the
full $K$. Notice that, by definition of $\phi$, $\mathbb{C}(X)$ maps
isomorphically on $k$.
\end{proof}

Since $\Aut(C_K/\CX)$ acts on $C_K(X)$ (via its action on
coefficients),  we can consider its action on $K$.

\begin{lem} \label{action}
\begin{enumerate}
\item The action of $\Aut(C_K/\CX)$ on $K$ extends the natural action of $\Aut(C_K/\CX)$
  on $C_K$. Moreover the action of $\Aut(C_K/\CX)$ on $K$ is trivial on $k$.
\item $K^{\Aut(C_K/\CX)}=k$.
\item The action of $\Aut(C_K/\CX)$ on $K$ commutes with the action of
  $\sigma_q$.
\end{enumerate}
\end{lem}
\begin{proof} $1.$ This  comes from the  definition of the action of $\Aut(C_K/\CX)$ on
  $K$.  Because  $\Aut(C_K/\CX)$ acts trivially on $\mathbb{C}(X)$, its action
  on $k$ is also trivial.\\[0.1in]
$2.$ Because of Lemma~\ref{galois}, $C_K^{\Aut(C_K/\CX)}=\mathbb{C}$. Thus, by construction $K^{\Aut(C_K/\CX)}=k$.\\[0.1in]
$3.$  Let $ i$ be a natural integer and $f(X) = c X^i$ where $c \in
  C_K$. Then
$$ \tau ( \sigma_q (f)) =\tau (c q^i X^i)= \tau(c) q^i X^i
=\sigma_q(\tau (f))$$ with $\tau \in \Aut(C_K/\CX)$. Thus, the
action of $\Aut(C_K/\CX)$ commutes with $\sigma_q$ on $C_K[X]$. It
therefore commutes on $C_K(X)=K$.
\end{proof}
Before we finally define   the action of $\Aut(C_K/\CX)$ on $\lbrace
\lbrace M_K \rbrace \rbrace$, we need one more definition.
\begin{defin}
Let $F$ be a field of caracteristic zero and $V$ be a $F$-vector
space of finite dimension over $F$. We denote by $Constr_F(V)$ any
\textit{construction of linear algebra} applied to $V$ inside
$Vect_{F}$, that is to say any  vector space over $F$ obtained by
tensor products over $F$, direct sums, symetric and antisymetric
products on $V$ and its dual $V^* := Hom_{F-lin}(V, F)$.
\end{defin}
\begin{lem}
Let $V$ be a vector space of finite dimension over $k$ (resp. over
$\mathbb{C}$). Then, $Constr_k(V) \otimes_k K =Constr_{K}(V \otimes
K)$ (resp. $Constr_{\mathbb{C}}(V) \otimes C_K =Constr_{C_K}(V
\otimes C_K)$). In other words, the \textit{constructions of linear
algebra} commute with the scalar extension.
\end{lem}
\begin{proof} Consider for instance $Constr_k(V)= Hom_{k-lin}(V,k)$.  Because $V$ is of finite dimension over $k$, we have $Hom_{k-lin}(V,k) \otimes_k
K=Hom_{K-lin}(V \otimes_k K, K)$. \end{proof}
\noindent To define the action of $\Aut(C_K/\CX)$ on $\lbrace
\lbrace M_K \rbrace \rbrace$ we note that for any object $N$ of
$\lbrace \lbrace M_K \rbrace \rbrace$, there exists, by definition,
a construction $M'=Constr_k(M)$ such that $N \subset M' \otimes_k
K$. Let now $M'=Constr_k(M)$ be a construction of linear
  algebra  applied to $M$. The Galois group $\Aut(C_K/\CX)$ acts on
  $M'_K =M' \otimes_k K$ via the
semi-linear action $(\tau \rightarrow id \otimes \tau)$. It
therefore permutes the objects of $\lbrace \lbrace M_K \rbrace
\rbrace$. It remains to prove that this permutation is well defined
 and is independent of the choice of the
construction in which these objects lie.  If there exist $M_1$ and
$M_2$ two objects of $Constr_k(M)$ such that $N \subset M_1
\otimes_k K$ and $N \subset M_2 \otimes_k K$. Then, by a diagonal
embedding  $N \subset (M_1 \oplus M_2)\otimes_k K$. The action of
$\Aut(C_K/\CX)$ on
  $(M_1\oplus M_2) \otimes_k K$ is the direct sum of the action of $\Aut(C_K/\CX)$  on
  $M_1 \otimes_k K$ with the action of $\Aut(C_K/\CX)$  on
  $M_2 \otimes_k K$.This shows that the restriction of the action of $\Aut(C_K/\CX)$  on
  $M_1 \otimes_k K$ to $N$ is the same as the restriction of the action of $\Aut(C_K/\CX)$  on
  $M_2 \otimes_k K$ to $N$. Thus, the permutation is independent of the choice of the
construction in which these objects lie.

\subsection{Another fiber functor $\tilde{\omega}_M$ for  $\lbrace \lbrace M_K \rbrace
\rbrace$}

We now extend $\omega_M$ to a fiber functor $\tilde{\omega}_M$ on
the category $\lbrace \lbrace M_K \rbrace \rbrace$. For this
purpose, we appeal to Proposition~\ref{p2} to conclude that {\em if
  $R$ be  a  \textit{Picard-Vessiot ring} for   $M$ over  $k$ and  $\sigma X=A X,
A\in \GL_n(k)$ be an equation of  $M$ over  $k$. If
$R=K[Y,\frac{1}{det(Y)}]/I$ where  $Y$ is an  $n \times n$ matrix of
indeterminates,  $\sigma Y =AY$ and  $I$ is a maximal
$\sigma$-ideal, then $R_K=R\otimes_k K$ is a  \textit{weak
Picard-Vessiot ring} for
$M_K$ over $K$.}\\[0.1in]
We then have the following proposition-definition:

\begin{prop}\label{functor}
For any object  $N$  of  $\lbrace \lbrace M_K \rbrace \rbrace$ let
$$\tilde{\omega}_M(N)= \mathrm{Ker} (\sigma -Id, R_K \otimes_{K}
N).$$ Then  $\tilde{\omega}_M : \lbrace \lbrace M_K \rbrace \rbrace
\rightarrow Vect_{C_K}$ is a faithful exact, $C_K$-linear tensor
functor. Moreover, $\tilde{\omega}_M(N \otimes K)=\omega_M(N)
\otimes C_K$ for every $N \in \lbrace \lbrace M \rbrace \rbrace$.
\end{prop}
\begin{proof} Because of the existence of  a fundamental matrix with coefficients in
$R_K$, $\tilde{\omega}_M(M_K)$ satisfies $R_K \otimes_{K_K} M_K =
R_K \otimes_{C_K} \tilde{\omega}_M(M_K)$. Let $\sigma X=A X, A\in
\GL_n(k)$ be an equation of  $M$ over  $k$  and
$R=k[Y,\frac{1}{det(Y)}]/I$ be its corresponding Picard-Vessiot ring
over $k$. Let $M'$ be a  construction of linear
  algebra  applied to $M$ over $k$. Then $R_K$ contains a fundamental
matrix of $M'\otimes K$. This comes from the fact that an equation
of $M'$ is obtained from the same   construction of linear
  algebra applied to $A$. Moreover, if $N \in \lbrace \lbrace M_K
\rbrace \rbrace$, then $R$ contains also a fundamental matrix for
$N$. Indeed,  there exists  $M'$, a  construction of linear
  algebra applied to $M$ over $k$, such that $N \subset
M'\otimes K$. Now, $R_K$ contains the entries of a fundamental
solution matrix of $N$ and this matrix is invertible because its
determinant divides the determinant of a fundamental matrix of
solutions of $M' \otimes K$. Thus,  $R_K \otimes_{K} N = R_K
\otimes_{C_K} \tilde{\omega}_M(N)$. We deduce from this fact, that
$\tilde{\omega}_M$ is a faithful, exact, $C_K$-linear tensor
functor.\\[0.1in]
For every $N \in \lbrace \lbrace M \rbrace \rbrace$, we have a
natural inclusion of $C_K$-vector spaces of solutions
$\omega_M(N)\otimes C_K \subset \tilde{\omega}_M(N \otimes K)$.
Since their dimensions over $C_K$ are both equal to the dimension of
$N$ over $k$,  they must coincide. \end{proof}

\subsection{Comparison of the Galois groups}

Let $M'=Constr_K(M)$ be a construction of linear
  algebra  applied to $M$. The automorphism group $\Aut(C_K/\CX)$ acts on
  $\tilde{\omega}_M(M'_K) = \omega_M(M') \otimes_{\mathbb{C}}C_K$ via the
semi-linear action $(\tau \rightarrow id \otimes \tau)$. It
therefore permutes the objects of the tannakian category generated
by
  $\omega_M(M) \otimes_{\mathbb{C}} C_K$ inside $Vect_{C_K}$.
\begin{lem}\label{com2}
 Let $N$ be an object of  $\lbrace \lbrace
M_K \rbrace \rbrace$ and $\tau$ be an element of $\Aut(C_K/\CX)$.
Then, for the actions of $\Aut(C_K/\CX)$ defined as above and in
Section~\ref{autsec}, we have:
$$  \tau (\tilde{\omega}_M(N)) = \tilde{\omega}_M(\tau N) $$
(equality inside $\omega_M(M')\otimes _{\mathbb{C}}C_K$ for any $M'
= Constr_k(M)$ such that $N \subset M' \otimes K$.)

\end{lem}
\begin{proof} Let $M' = Constr_K{M}$ be such that $N \subset M'
\otimes_k K$ and consider the action of  $\Aut(C_K/\CX)$ on $R \otimes_k(M' \otimes_k K) $ defined by $id \otimes  id \otimes \tau$.\\[0.1in]
This allows us to consider the action of $\Aut(C_K/\CX)$ on $R_K
\otimes_{K} N = R \otimes_k N$. By definition, we have $\tau(R_K
\otimes_{K} N)=R \otimes_{k} (\tau(N))=R_K \otimes_{K} \tau(N)$ for
all $\tau \in \Aut(C_K/\CX)$. Moreover inside  $R \otimes_k
(M'\otimes K)$, the action of $\Aut(C_K/\CX)$ commutes with the
action of $\sigma_q$ (see Lemma~\ref{action}). Therefore
$$\tau(\mathrm{Ker}(\sigma_q -Id, R_K \otimes_{K} N))=\mathrm{Ker}(\sigma_q -Id,
R_K \otimes_{K} \tau(N)).$$\end{proof}

The next proposition  is Corollary~\ref{cor1}, but we shall now give
a tannakian proof of it, following the proof of
(\cite{katz_calculations}, Lemma 1.3.2).
\begin{prop}\label{prop3.8}
$Aut^{\otimes}( \omega_M) \otimes C_K =
  Aut^{\otimes} (\tilde{\omega}_M).$
\end{prop}
\begin{proof} By definition, $Aut^{\otimes}( \tilde{\omega}_M) = Stab (
 \tilde{\omega}_M(W), \ W \in \lbrace \lbrace
M_K \rbrace \rbrace )$ is the stabilizer inside
 $\mathcal{G}l(\tilde{\omega}_M(M_K))=\mathcal{G}l(\omega_M(M))
 \otimes C_K$ of
  the fibers of all the  sub-equations  $W$ of $M_K$. Similarly,  $Aut^{\otimes}( \omega_M) = Stab (
 \omega_M(W)), \ W \in \lbrace \lbrace
M\rbrace \rbrace )$, so that   the following inclusion holds:\\
$$Aut^{\otimes}( \tilde{\omega}_M) \subset Aut^{\otimes}(
  \omega_M) \otimes C_K.$$
 The semi-linear action of $\Aut(C_K/\CX) $ permutes the sub-$\mathcal{D}_{K}$-modules $W$ of $ \lbrace \lbrace
M_K \rbrace \rbrace$ and  the fixed field of $C_K$ of  $\Gamma_E$ is
$\mathbb{C}$ (see Lemma~\ref{galois}$.1)$).  Therefore
$Aut^{\otimes}
 ( \tilde{\omega}_M)$ \textit{is defined} over  $\mathbb{C}$,
i.e., it is of the form  $G \otimes
 C_K$ for a unique subgroup $G \subset Aut^{\otimes}( \omega_M)$. By
 Chevalley's theorem,  $G$ is defined as the stabilizer of one
 $\mathbb{C}$-subspace  $V$ of  $\omega_M(M')$ for some construction
 $M' = Constr_{k} (M)$.\\[0.1in]
We must show that $V$ is stable under $Aut^{\otimes}(
  \omega_M)$, i.e., we must show that $V$ is of the form
  $\omega_M(N)$ for $N \in \lbrace \lbrace M \rbrace \rbrace$. Because
  $G \otimes C_K =Aut^{\otimes}( \tilde{\omega}_M)$ leaves $V
  \otimes C_K$ stable, we know that there exists $N \in
  \lbrace \lbrace M_K \rbrace \rbrace$ with
  $\tilde{\omega}_M(N) = V \otimes C_K$. For any $\tau \in
  \Aut(C_K/\CX)$, $$\tilde{\omega}_M( N) = V \otimes C_K = \tau (V \otimes C_K)=\tau
  (\tilde{\omega}_M(N))=\tilde{\omega}_M(\tau N),$$
 in view of Lemma~\ref{com2}. We therefore deduce from Proposition~\ref{functor} that $\tau
  N=N$ for any $\tau \in \Aut(C_K/\CX)$. Consequently,
  $N$ is \textit{defined over} $K$ (see Lemma~\ref{action}$.3)$), i.e., it is of the form
  $N \otimes K$, where $N \in \lbrace \lbrace M \rbrace \rbrace$.\end{proof}
We need to define one more functor before we finish the proof of
Theorem~\ref{ext}.
\begin{prop}\label{prag}
Let \begin{equation} \label{eqn:5} \sigma_q Y= A Y \end{equation} be
an equation of $M$ with $A \in \GL_n(K)$. There exists a fundamental
matrix of solutions $U$ of (\ref{eqn:5}) with coefficients in the
field $\calM(\mathbb{C}^*)$ of functions meromorphic on $\CX^*$.
Moreover, if $V$ is another fundamental matrix of solutions of
(\ref{eqn:5}), there exists $P \in
\GL_n(C_K)$ such that $U = PV$. \\[0.05in]
Let  $L$  be the subfield of $\calM(\mathbb{C}^*)$ generated over
$K$ by the entries of $U$.   For any object  $N$  of  $\lbrace
\lbrace M_K \rbrace \rbrace$ let  $${\omega}_{M_K}(N)= {\rm Ker}
(\sigma -Id, L \otimes N).$$ Then  ${\omega}_{M_K} : \lbrace \lbrace
M_K \rbrace \rbrace \rightarrow Vect_{C_K}$ is a faithful exact,
$C_K$-linear tensor functor.
\end{prop}
\begin{proof} For the existence of $U$ see \cite{praagman} Theorem $3$. Since the field of constants of $L$ is $C_K$, $L$ is a \textit{weak Picard Vessiot} field for $M_K$, and the
proof that $\omega_{M_K}$ is a fiber functor on  $\lbrace \lbrace
M_K \rbrace \rbrace$ is   the same  as that of
Proposition~\ref{functor}.\end{proof}
We now turn to the
\begin{proofofthm}
By Propositions~\ref{functor} on the one hand and \ref{prag}   on
the other hand, there exists two fiber functors $\tilde{\omega}_M$
and $\omega_{M_K}$  on $\lbrace \lbrace M_K \rbrace \rbrace$ which
is a neutral tannakian category over $C_E$. A fundamental theorem of
Deligne (\cite{deligne_milne}, Theorem 3.2)  asserts
 that for
 any field $C$ of caracteristic zero,  two fiber functors of a neutral tannakian
category over   $C$  become isomorphic over the algebraic closure of
$C$. Taking $C=C_K$ and combining this with
Proposition~\ref{prop3.8}, we therefore have
   $$ Aut^{\otimes}( \omega_M) \otimes \overline{C_K} \simeq  Aut^{\otimes}( \omega_{M_K}) \otimes \overline{C_K}.$$ \end{proofofthm}
To show the connection between Theorem~\ref{ext} and
Theorem~\ref{thethm} we must show that the group of difference $k$
(rep. $K$)-automorphisms of $R$ (resp. $F$) can be identified with
the $\mathbb{C}$ (resp. $C_K$)-points of $Aut^{\otimes}( \omega_M)$
(resp. $Aut^{\otimes}( \omega_{M_K}$). In the first case, this has
been shown in Theorem 1.32.3 of \cite{PuSi}; the second case can be
established  in a similar manner.
This enables us to  deduce, in our special  case, Theorem~\ref{thethm} from Theorem \ref{ext}.\\[0.1in]
We conclude with an example to show that these considerations can be
used to show the algebraic independence of certain classical
functions.
\begin{ex} {\em Consider the $q$-difference equation
\begin{equation}\label{eqn:4} \sigma_q(y) = y + 1. \end{equation} In
(\cite{PuSi}, Section 12.1) the authors denote by $l$ the formal
solution of \ref{eqn:4}, i.e. the formal $q$-logarithm.  It is
easily seen that the Galois group, in the sense of \cite{PuSi}, of
(\ref{eqn:4})  is equal to $(\mathbb{C},+)$ and therefore that the
dimension of the Galois group ${G_R}_{/\mathbb{C}}$ is equal to $1$.
We deduce from Theorem 2.9 that the dimension of the Galois group
${G_S}_{/C_K}$ is also equal to $1$. In particular,
the field generated  over $K$ by the meromorphic solutions of (\ref{eqn:4}) has transcendence degree $1$ over $K$.\\[0.1in]
The classical Weiestrass $\zeta$ function associated to the elliptic
curve $E =\mathbb{C}^* / q^{\mathbb{Z}}$  satisfies the equation
(\ref{eqn:4}). Therefore, if $\wp$  is the Weierstrass function of
$E$, we obtain that  $\zeta(z)$ is transcendental over the field
$\mathbb{C}(z, \wp(z))$.} \end{ex}

\section{A Model-Theoretic Point of View}\label{mt}

\subsection{Preliminary model-theoretic definitions and results}\label{mt1}
\begin{defin} Let $K$ be a difference field with automorphism
  $\sigma$.\begin{itemize} 
\item[1.]$K$ is {\it generic} iff 
\begin{quotation}\noindent $(*)$  every
{\bf finite} system of difference equations with coefficients in $K$
and which has a solution in a difference field containing $K$,
already has a solution in $K$.\end{quotation}
\item[2.]A {\em finite $\si$-stable extension} $M$ of $K$ is a finite separably
   algebraic extension of $K$ such that $\si(M)=M$.
\item[3.]The {\em core of $L$ over $K$}, denoted by ${\rm Core}(L/K)$, is the
   union of all finite $\si$-stable extensions of $K$ which are
   contained in $L$.
\end{itemize}
\end{defin}
One of the difficulties with difference fields, is that there are
usually several non-isomorphic ways of extending the automorphism to
the algebraic closure of the field. An important result of Babbitt
(see \cite{cohn}) says that once we know the behaviour of $\si$ on
${\rm Core}(\overline{K} /K)$, then we know how $\si$ behaves on the
algebraic  closure $\overline{K}$ of $K$.

Fix an infinite cardinal $\kappa$ which is  larger than all the
cardinals of structures considered (e.g., in our case, we may take
$\kappa=|\CX|^+=(2^{\aleph_0})^+$). In what follows we will work in
a generic difference field $\calu$, which we will assume {\em
sufficiently saturated}, i.e., which has the following properties:
\begin{itemize}
\item[(i)] $(*)$ {\em above holds for every system of difference equations {\bf of
size} $<\kappa$ (in infinitely many variables). }
\item[(ii)] {\em (1.5 in \cite{ChHr1}) If $f$ is an isomorphism between two algebraically closed
  difference subfields of $\calu$ which are of cardinality $<\kappa$,
  then $f$ extends to an automorphism of $\calu$. }
\item[(iii)] {\em Let $K\subset L$ be difference fields of cardinality
  $<\kappa$, and assume that $K\subset \calu$. If every finite
  $\si$-stable extension of $K$ which is contained in $L$ $K$-embeds in
  $\calu$, then there is
 a $K$-embedding of $L$ in $\calu$.}
\end{itemize}
Note that the hypotheses of (iii) are always verified if $K$ is an
algebraically closed subfield of $\calu$. If $K$ is a difference
field containing the algebraic closure $\overline{\QX}$ of $\QX$,
then $K$ will embed into $\calu$, if and only if the difference
subfield $\overline{\QX}$ of $K$ and the difference subfield
$\overline{\QX}$ of $\calu$ are isomorphic. This might not always be
the case. However, every difference field embeds into some
sufficiently saturated generic difference field.

Let us also recall the following result  (1.12 in
  \cite{ChHr1}):  Let $n$ be a positive integer, and consider the
  field $\calu$
  with the automorphism $\si^n$. Then $(\calu,\si^n)$ is a generic difference field, and satisfies the
saturation properties required of $(\calu,\si)$.

\medskip\noindent
{\bf Notation}. We use the following notation. Let $R$ be a
difference ring. Then, as in the previous sections,  $C_R$ denotes
the field of ``constants'' of $R$, i.e., $C_R=\{a\in R\mid
\si(a)=a\}$. We let  $D_R = \{ a \in R \ | \ \sigma^m(a) = a \mbox{
for some }
 m\neq 0\}$. Then $D_R$ is a difference subring of $R$, and if $R$ is a field,
 $D_R$ is
 the relative algebraic closure of $C_R$ in $R$. We let $D'_R$ denote
 the difference ring with same underlying ring as $D_R$ and on 
 which $\si$ acts trivially. Thus $C_\calu$ is a
pseudo-finite field (see 1.2 in \cite{ChHr1}), and $D_\calu$ is its
algebraic closure (with the action of $\si$), $D'_\calu$ the
algebraic closure of $C_\calu$ on which $\si$ acts trivially.

Later we will work with powers of $\si$, and will  write
$Fix(\si^n)(R)$ for $\{a\in R\mid \si^n(a)=a\}$
 so that no
confusion arises. If $R=\calu$, we will simply write $Fix(\si^n)$.
Here are some additional properties of $\calu$ that we
will use.\\[0.1in]
Let $K\subset M$ be difference subfields of $\calu$, with $M$
algebraically closed, and let $a$ be a tuple of $\calu$. By 1.7 in
\cite{ChHr1}:

(iv) If the orbit of $a$ under $\Aut(\calu/K)$ is finite, then
$a\in\overline{K}$ (the algebraic closure of $K$). \\[0.1in]
We already know that every element of $\Aut(M/KC_M)$ extends to an
automorphism of $\calu$. More is true: using  1.4, 1.11 and Lemma 1
in the appendix of \cite{ChHr1}:

(v) every element of
$\Aut(M/KC_M)$ lifts to an element of $\Aut(\calu/KC_\calu)$.\\[0.1in]
Recall that a definable subset $S$ of $\calu^n$ is {\em stably
embedded} if whenever $R\subset \calu^{nm}$ is definable with parameters
 from $\calu$, then
$R\cap S^m$ is definable using parameters from $S$. An important
result (\cite{ChHr1} 1.11)) shows that $C_\calu$ is stably embedded.
Let $d\geq 1$. Then, adding parameters from $Fix(\si^d)$,  there is
a definable isomorphism between $Fix(\si^d)$ and $C_\calu^d$. Hence,

(vi) for every $d>0$, $Fix(\si^d)$ is  stably embedded, and

(vii) if $\theta$ defines an automorphism of $D_\calu$ which is the
identity on $D_M$, then $\theta$ extends to an automorphism of
$\calu$
which is the identity on $M$.\\[0.1in]
We also need the following lemma. The proof is rather
model-theoretic and we refer to the Appendix of \cite{ChHr1} for the
definitions and results.  Recall that if $K$ is a difference
subfield of $\calu$, then its definable closure, $\dcl(K)$, is the
subfield of $\calu$ fixed by $\Aut(\calu/K)$. It is an algebraic
extension of $K$, and is the subfield of the algebraic closure
$\overline{K}$ of $K$ which is fixed by the subgroup $\{\tau\in
\calG al(\overline{K}/K)\mid \si^{-1}\tau\si=\tau\}$.
\begin{lem}\label{zlem1}
 Let $K$ be a difference field, and $M$ be a finite
$\si$-stable extension of $KC_\calu$. Then $M\subset 
\overline{K}D_\calu$, i.e., there is some finite $\si$-stable extension
$M_0$ of $K$ such that $M\subset M_0D_\calu $.\end{lem}

\begin{proof}Fix an integer $d\geq 1$. Then, in the difference field
  $(\calu,\si^d)$, $Fix(\si^d)$ is stably embedded, 
$\dcl(\overline{K})=\overline{K}$ and 
  $\dcl(Fix(\si^d))=Fix(\si^d)$. Denoting types in $(\calu,\si^d)$
  by $tp_d$, this implies
$$tp_d(\overline{K}/\overline{K}\cap
Fix(\si^d))\vdash tp_d(\overline{K}/Fix(\si^d)).\eqno{(\sharp)}$$

Assume by way of contradiction that $KC_\calu$ has a finite
$\si$-stable extension $M$ which is not contained in
$\overline{K}D_\calu$. We may assume that $M$ is Galois over $\overline{K}C_\calu$
(see Thm 7.16.V in \cite{cohn}), with Galois group $G$. Choose 
$d$ large enough so that
$\si^d$ commutes with all elements of $G$, and $M=M_0D_\calu$, where
$M_0$ is Galois over $\overline{K}Fix(\si^d)$. Then there are
several non-isomorphic ways of extending $\si^d$ to $M$. As
$tp_d(\overline{K}/Fix(\si^d))$  describes in particular the
$\overline{K}Fix(\si^d)$-isomorphism type of the $\si^d$-difference field $M$, 
this
contradicts $(\sharp)$ (see Lemmas 2.6 and 2.9 in
\cite{ChHr1}).\end{proof}

\subsection{The Galois group}\label{zpv}
From now on, we assume that all fields are of characteristic $0$.
Most of the statements below can be easily adapted to the positive
characteristic case. Let $K$ be a difference subfield of $\calu$,
$A\in \GL_n(K)$, and consider the set $\calS=\calS(\calu)$ of
solutions of the equation
$$\si(X)=AX,\ \det(X)\neq 0.$$
Consider $H=\Aut(K(\calS)/KC_\calu)$. We will call $H$ the {\em
Galois group of
    $\si(X)=AX$ over $KC_\calu$}\footnote{\ Warning: This is not the
    usual Galois group
      defined by model theorists, please see the discussion in subsection
      \ref{mt4}.}.

Then $H$ is the set of $C_\calu$-points of some algebraic group
$\HX$ defined over $KC_\calu$. To see this, we consider the ring
$R=K[Y, \det(Y)^{-1}]$ (where $Y=(Y_{i,j})$ is an $n\times n$ matrix
of indeterminates), extend $\si$ to $R$ by setting $\si(Y)=AY$, and
let $L$ be the field of fractions of $R$. Then $L$ is a regular
extension of $K$, and there is a $K$-embedding $\varphi$ of $L$ in
$\calu$, which sends $C_L$ to a subfield of $C_\calu$, and $D_L$ to
a subfield of $D_\calu$. We let $T=\varphi(Y)$. Then every element
$g\in H$ is completely determined by the matrix $M_g=T^{-1}g(T)\in
\GL_n(C_\calu)$, since if $B\in \calS$, then
$B^{-1}T\in\GL_n(C_\calu)$. Moreover, since $KC_{\varphi(L)}(T)$ and
$KC_\calu$ are linearly disjoint over $KC_{\varphi(L)}$, the
algebraic locus $W$ of $T$ over $KC_\calu$ (an algebraic subset of
$\GL_n$) is defined over $KC_{\varphi(L)}$, and $H$
  is the set of elements of $\GL_n(C_\calu)$ which leave $W$ invariant. It is
  therefore the set of $C_\calu$-points of an algebraic group $\HX$,
  defined over $KC_{\varphi(L)}$. We let $\HX'$ denote the Zariski
  closure of $\HX(C_\calu)$. Then
  $\HX'$ is defined over $C_\calu$, and it is also clearly defined over
  $\overline{K\varphi(C_L)}$, so that it is defined over $C_\calu\cap
  \overline{K\varphi(C_L)}=C_\calu\cap \overline{\varphi(C_L)}$.

\begin{prop}\label{zggp1}Let $\HX^0$ denote the connected component of
  $\HX$, and let $M_0$ be the relative algebraic closure of
  $K\varphi(C_L)$ in $\varphi(L)$, $M$ its Galois closure over
  $K\varphi(C_L)$.
\begin{itemize}\item[1.]$\dim(\HX)={\rm tr.deg}(L/KC_L)$.
\item[2.]$M_0$ is a finite $\si$-stable extension of
  $K\varphi(C_L)$ and $[\HX:\HX^0]$ divides $[M:K\varphi(C_L)]$
\item[3.] $[\HX':\HX^0]=[\HX(C_\calu):\HX^0(C_\calu)]$ equals 
the number of left cosets of 
  $\calG al(M/M_0)$ in $\calG al(M/K\varphi(C_L))$ which are invariant
  under the action of $\si$ by conjugation.
\item[4.]If the algebraic closure of $C_K$ is contained in $C_\calu$,
  then the element $\si\in \Aut(D_L/C_L)=\calG al(D_L/C_L)$ extends to an element of
 $\Aut(KC_\calu(T)/KC_\calu)$.
\end{itemize}\end{prop}
\begin{proof} 1. Choose another $K$-embedding $\varphi'$ of $L$ into
  $\calu$ which extends $\varphi$ on the relative algebraic closure of
  $KC_L$ in $L$, and is such that
  $\varphi'(L)$ and $\varphi(L)$ are linearly disjoint
  over $M_0$. Then $B=\varphi'(Y)^{-1}T\in \HX(C_\calu)$, and ${\rm
  tr.deg}(\varphi(KC_L)(B))/\varphi(KC_L))={\rm tr.deg}(L/KC_L)$. Thus
  $\dim(\HX)={\rm tr.deg}(L/KC_L)$.\\[0.1in]
2. As $M_0\subset K\varphi(L)$, we obtain that $[M_0:K\varphi(C_L)]$
is
  finite and $\si(M_0)=M_0$. Furthermore, $\si(M)=M$ (see Thm 7.16.V in
  \cite{cohn}). The algebraic group $\HX$ is defined as the set of matrices of
  $\GL_n$ which leaves the algebraic set $W$ (the algebraic
  locus of $T$ over $K\varphi(C_L)$) invariant.

Hence $\HX^0$ is the
  subgroup of $\HX$ which leaves all absolutely irreducible components of $W$
  invariant. Its index in $\HX$ therefore divides $[M:K\varphi(C_L)]$.
\\[0.1in]
3. The first equality follows from the fact that $\HX^0(C_\calu)$
and
  $\HX'(C_\calu)$ are Zariski dense in $\HX^0$ and $\HX'$
  respectively. Some of the (absolutely irreducible) components of $W$
  intersect $\calS$ in the
  empty set. Indeed, let $W_0$ be the component of $W$ containing $T$,
  let $W_1$ be another component of $W$ and $\tau\in
  {\calG}al(M/K\varphi(C_L))$ such that $W_1=W_0^\tau$. Then $W_1$ is
  defined over $\tau(M_0)$. If $\tau$ defines a (difference-field)
  isomorphism between $M_0$ and $\tau(M_0)$, then $\tau$ extends to an
  isomorphism between $K\varphi(L)$  and a regular extension of
  $K\varphi(C_L)\tau(M_0)$, and
  therefore $W_1\cap \calS\neq \emptyset$. Conversely, if $B\in W_1\cap
  \calS$, then $B^{-1}T\in \HX(C_\calu)$, so that $B$ is a generic of
  $W_1$.  The difference fields
  $K\varphi(C_L)(B)$ and $K\varphi(L)$ are therefore isomorphic (over
  $K\varphi(C_L)$), and $\tau(M_0)\subset K\varphi(C_L)(B)$. Hence the
  difference subfields $M_0$ and $\tau(M_0)$ of $M$ are
  $K\varphi(C_L)$-isomorphic. One verifies that $M_0$ and $\tau(M_0)$
  are isomorphic over $K\varphi(C_L)$ if and only if $\si^{-1} \tau^{-1}
  \si\tau\in \calG al(M/M_0)$, if and only if the coset $\tau\calG
  al(M/M_0)$ is invariant under the action of $\si$ by conjugation.\\[0.1in]
4. We know that the algebraic closure $\overline{K}$ of $K$ and
  $D_\calu$ are linearly disjoint over
  $C_{\overline{K}}=\overline{C_K}$. Let $a\in \varphi(D_L)$ generates
  $\varphi(D_L)$ over $\varphi(C_L)$. By \ref{mt1}(vi),
  $tp(a/\overline{K}C_\calu)=tp(\si(a)/\overline{K}C_\calu)$, and therefore there is
 $\theta\in\Aut(\calu/\overline{K}C_\calu)$ such that
  $\theta(a)=\si(a)$. Thus $T^{-1}\theta(T)\in H$.
\end{proof}
\begin{remarks}\label{zgprem}\begin{itemize}
\item[1.] Even when the algebraic closure of $C_K$ is
  contained in $C_\calu$, we still cannot in general conclude that
  $\HX'=\HX$.
\item[2.]
The isomorphism type of the algebraic group $\HX$ only depends on
the
  isomorphism type of the difference field $K$ (and on the matrix $A$). The
  isomorphism type of the algebraic group $\HX'$ does however depend on
  the embedding of $K$ in $\calu$, that is, on the isomorphism type of
  the difference field ${\rm Core}(\overline{K}/K)$. Indeed, while we
  know the isomorphism type of the difference field $M_0$ over $K\varphi(C_L)$, we do not
  know the isomorphism type of the difference field $M$ over
  $K\varphi(C_L)$, and in view of \ref{zggp1}.3, if $\calG
  al(M/K\varphi(C_L))$ is not abelian, it may happen that
  non-isomorphic extensions of $\si$ to $M$ yield different Galois
  groups.
\item[3.]
If  the action of $\si$ on $\Gal({\rm
Core}(\overline{K}/K)/K)$ is trivial and $\Gal({\rm
Core}(\overline{K}/K)/K)$ is abelian, then $$\HX=\HX'\quad \hbox{and}\quad
[\HX:\HX^0]=[M_0:K\varphi(C_L)].$$ Indeed, by Lemma \ref{zlem1}, 
$M_0$ is Galois over $K\varphi(C_L)$ with abelian Galois group $G$ and
$\si$ acts trivially on $G$. The result follows by \ref{zggp1}.3. 
Thus we
  obtain equality of $\HX$ and $\HX'$ in two important classical cases:
\begin{itemize}
\item[a.]$K=\CX(t)$, $C_K=\CX$ and $\si(t)=t+1$.
\item[b.]$K=\CX(t)$, $C_K=\CX$ and $\si(t)=qt$ for some $0\neq q\in\CX$,
  $q$ not a root of unity.\end{itemize}
\item[4.]If $B\in \calS$, then the above construction can be repeated,
  using $B$ instead of $T$. We then obtain an algebraic group $\HX_1$,
  with $\HX_1(C_\calu)\simeq \Aut(KC_\calu(\calS)/KC_\calu)$. Since
  $KC_\calu(B)=KC_\calu(T)$, the algebraic groups
  $\HX_1$ and $\HX$ are isomorphic (via $B^{-1}T$).
\item[5.]In the next subsection, we will show that the algebraic group
  $\HX$ and the algebraic group $G_{R'}$ introduced in section
  \ref{galoissec} are isomorphic when $C_{R'}=C_K=D_K$.
\end{itemize}\end{remarks}
\subsection{More on Picard-Vessiot rings}\label{zpvr}

Throughout the rest of this section, we fix a difference ring $K$,
some $A\in \GL_n(K)$, $R=K[Y,\det(Y)^{-1}]$ as above, with
$\si(Y)=AY$, and $R'=R/q$ a Picard-Vessiot ring for $\si(X)=AX$ over
$K$. We denote the image of $Y$ in $R'$ by $y$. We keep the notation
introduced
in the previous subsections.\\[0.1in]
If $q$ is not a prime ideal, then there exists $\ell$ and a prime
$\si^\ell$-ideal $p$ of $R$ which is a maximal $\si^\ell$-ideal of
$R$, such that $q=\bigcap_{i=0}^{\ell
  -1}\si^i(p)$, and $R'\simeq \oplus_{i=0}^{\ell -1}R_i$, where
$R_i=R/\si^i(p)$ (see  Corollary 1.16 of \cite{PuSi}. One
  verifies that the second proof does not use the fact that $C_K$ is
  algebraically closed). Thus the $\si^\ell$-difference ring
$R_0$ is a Picard-Vessiot ring for
  the difference equation $\si^\ell(X)=\si^{\ell-1}(A)\cdots\si(A)AX$ over
  $K$.  We denote $\si^{\ell-1}(A)\cdots\si(A)A$ by $A_\ell$. \\[0.1in]
We will identify $R'$ with $\oplus_{i=0}^{\ell
    -1}R_i$, and denote by $e_i$ the primitive
  idempotent of $R'$ such that $e_iR'=R_i$. Then $e_i=\si^i(e_0)$. We
  will denote by $R^*$ the ring of quotients
  of $R'$, i.e., $R^*=\oplus_{i=0}^{\ell -1}R_i^*$, where $R_i^*$ is the
  field of fractions of $R_i$. The difference ring $R^*$ is also called
  the {\em total
    Picard-Vessiot ring of $\si(X)=AX$ over $K$}. There are two
  numerical invariants associated to $R'$: the number $\ell=\ell(R')$,
  and the number $m(R')$ which is the product of $\ell(R')$ with
  $[D_{R^*_0}:D_KC_{R^*_0}]$. We call $m(R')$ the {\em $m$-invariant of
    $R'$}. We will be considering other Picard-Vessiot
  rings for $\si(X)=AX$, and will use this notation for them as well.\\[0.1in]
Recall that the {\em Krull dimension} of a ring $S$ is the maximal
integer $n$ (if it exists) such that there is a (strict) chain  of
prime ideals of $S$ of length $n$. We denote it by $\Krd(S)$. If $S$
is a domain, and is finitely generated over some subfield $k$, then
$\Krd(S)$ equals the transcendence degree over $k$ of its field of
fractions. Observe that if $S$ is a domain of finite Krull
dimension, and $0\neq I$ is an ideal of $S$, then
$\Krd(S)>\Krd(S/I)$. Also, if $S=\oplus_i S_i$, then
$\Krd(S)=\sup\{\Krd(S_i)\}$.
\begin{lem}\label{zpvlem2}
\begin{itemize}
\item[1.]
$C_{R'}$ is a finite algebraic extension of $C_K$, and is linearly
disjoint from $K$ over $C_K$ (inside $R'$).
\item[2.]If $C_{R'}\otimes_{C_K}D_K$ is a domain, then $R'$ is a Picard-Vessiot ring for $\si(X)=AX$ over $KC_{R'}$.
\end{itemize}
\end{lem}
\begin{proof}  1. We know by Lemma 1.7 of \cite{PuSi} that $C_{R'}$ is a
  field. Assume by way of contradiction that $C_{R'}$ and $K$ are not
  linearly disjoint over $C_K$, and choose $n$ minimal such that there
  are $a_1,\ldots,a_n\in C_{R'}$ which are $C_K$-linearly independent,
  but not $K$-linearly independent. Let $0\neq c_1,\ldots, c_n\in K$ be such that
  $\sum_{i=1}^n a_ic_i=0$. Multiplying by $c_1^{-1}$, we may
  assume $c_1=1$. Then $\si(\sum_{i=1}^n a_ic_i)=\sum_{i=1}^n
  a_i\si(c_i)=0$, and therefore $\sum_{i=2}^na_i(\si(c_i)-c_i)=0$. By
  minimality of $n$, all $(\si(c_i)-c_i)$ are $0$, i.e., all $c_i\in
  C_K$, which gives us a contradiction.

Observe that $e_0C_{R'}\subset Fix(\si^\ell)(R_0)$, and we may
therefore replace $R'$ by the domain $R_0$. Since $R_0$ is a
finitely generated $K$-algebra, we know that its Krull dimension
equals the transcendence degree over $K$ of its field of fractions.
Thus $R_0$ cannot contain a subfield which is transcendental over
$K$, i.e., the elements of
$Fix(\si^\ell)(R_0)$ are algebraic over $K$.  his furthermore implies that  $Fix(\si^\ell)(R_0)$ is an algebraic extension of $Fix(\si^\ell)(K)$. Since the latter field is an algebraic extension of $C_K$, we have the conclusion.\\[0.1in]
2. Our hypothesis implies that $K[C_{R'}]$ is a field. Hence $R'$ is
a simple difference ring containing $KC_{R'}$, and is therefore a
Picard-Vessiot ring for $\si(X)=AX$ over $KC_{R'}$. \end{proof}

\begin{lem}\label{zpvlem1}
\begin{itemize}
\item[1.] $C_{R'}=C_{R^*}$.
\item[2.]$Fix(\si^\ell)(e_0R^*)=e_0C_{R'}$.
\item[3.]$D_{R^*}=\oplus_{i=0}^{\ell -1}D_{e_iR^*}$.
\end{itemize}
\end{lem}
\begin{proof} 1.  If $c\in C_{R^*}$, then $c$ can be represented by some
  $\ell$-tuple $(\frac{a_0}{b_0},\ldots,\frac{a_{\ell -1}}{b_{\ell -1}})$,
  where $a_i,b_i\in R_i$, and $b_i\neq 0$. Thus the ideal $I=\{d\in
  R'\mid dc\in R'\}$ is a $\si$-ideal of $R'$ and contains the element
  $b=(b_0,\ldots,b_{\ell -1})\neq 0$. Since $R'$ is simple, $1\in I$,
  i.e., $c\in R'$.

\noindent 2. Assume that $a\in e_0R^*$ satisfies $\si^\ell(a)=a$.
Then $a=e_0a$,
 $\sum_{i=0}^{\ell
  -1}\si^i(e_0a)$ is fixed by $\si$, and therefore belongs to $C_{R'}$.
Hence $a\in e_0C_{R'}$. \\[0.1in]
3. If  $a\in R^*$ satisfies $\si^m(a)=a$ for some $m$, then
 $\si^{m\ell}(e_ia)=e_ia$. \end{proof}
\begin{remark}\label{zpvrem1} Observe that  $\ell$ and the isomorphism type
of the $K$-$\si^\ell$-difference algebra $R_0$ completely determine
the isomorphism type of the difference algebra $R'$. Indeed, for
each $i=1,\ldots,\ell-1$, one chooses a copy $R_i$ of the
 domain $R_0$, together with an isomorphism
$f_i:R_0\to R_i$ which extends $\si^i$ on $K$. This $f_i$ then
induces an automorphism $\si^\ell$ of $R_i$. One then defines $\si$
on $\oplus_{i=0}^{\ell -1}R_i$ by setting
$\si(a_0,\ldots,a_{\ell-1})=(f_1(a_0),f_2f_1^{-1}(a_1),\ldots,
\si^\ell f_{\ell -1}^{-1}(a_{\ell -1}))$.
\end{remark}
\begin{prop}\label{zp2}Let $K\subset K_1$ be difference fields of
  characteristic $0$ where $K_1=K(C_{K_1})$, and assume that $C_K=D_K$.  Then
  $R'\otimes_KK_1=\oplus _{i=1}^d R'_i$, where each $R'_i$ is a
  Picard-Vessiot ring for $\si(X)=AX$ over $K_1$, and $d\leq
  [C_{R'}:C_K]$. Moreover, each $R'_i$ has the same Krull-dimension and
  $m$-invariant as $R'$. \end{prop}

\begin{proof} Our assumption implies that
  $K\otimes_{C_K}C_{K_1}$ is a domain. Let $C$ be the relative algebraic closure of $C_K$ in $C_{K_1}$. Then $K(C)=K[C]$, and $R'\otimes_KK(C)\simeq
  R'\otimes_{C_K}C$.

Let $a\in C_{R'}$ be such that $C_{R'}=C_K(a)$ and let
  $f(X)\in C_{K}[X]$
  be its minimal polynomial over $C_K$.
Let $g_1(X),\ldots,g_d(X)$
  be the irreducible factors of $f(X)$ over $C$. Then
  $f(X)=\prod_{i=1}^d g_i(X)$, and $C_R'\otimes_{C_K}C\simeq
  \oplus_{i=1}^d C_i$,
  where $C_i$ is generated over $C$ by a root of
  $g_i(X)=0$. Indeed, identifying $C$ with $1\otimes C$, every prime ideal of $C_{R'}\otimes_{C_K}C$
  must contain some $g_i(a\otimes 1)$; on the other hand,  each
  $g_i(a\otimes 1)$ generates a maximal ideal of $C_{R'}\otimes_{C_K}C$.
Thus  $$ R'\otimes_{C_K}C\simeq R'\otimes_{C_{R'}}
(C_{R'}\otimes_{C_K}C)\simeq \oplus_{i=1}^dR'\otimes_{C_{R'}}C_i.$$
By Lemmas \ref{lem1.11} and \ref{zpvlem2}, each
  $R'\otimes_{C_{R'}}C_i=R'_i$ is a simple
difference ring, with field of constants $C_i$. Hence $R'_i$ is a
  Picard-Vessiot
ring for $\si(X)=AX$ over $KC$ (and also over $KC_i$). Note that
$d\leq \deg(f)=[C_{R'}:C_K]$, and that $\Krd(R'_i)=\Krd(R')$
(because
  $KC$ is algebraic over $K$, and $R'_i$ is finitely generated over $K$).

By Proposition \ref{p2}, $R'_i\otimes_{KC_i}K_1C_i$ is a
Picard-Vessiot
 ring. Because $C_i$ and $K_1$ are linearly disjoint over $C$, and
 $C_i$ is algebraic over $C$, $KC_i\otimes_{KC}K_1\simeq K_1C_i$, and
 therefore
$$R'_i\otimes_{KC}K_1\simeq  R'_i\otimes_{KC_i}K_1C_i.$$
This shows that $R'\otimes_{K}K_1$ is the direct sum of
Picard-Vessiot rings over $K_1$.

Identifying $C_{R'}$ with $e_jC_{R'}=C_{R_j}$, we obtain
$$R'_i=(\oplus_{j=0}^{\ell-1}R_j)\otimes_{C_{R'}}C_i\simeq
\oplus_{j=0}^{\ell-1}R_j\otimes_{C_{R'}}C_i.$$ Each $R_j$ being a
Picard-Vessiot ring for $\si^\ell(X)=A_\ell X$, we know by
Proposition \ref{p2} that $R_j\otimes_{C_{R'}}C_i$ is also a
Picard-Vessiot ring for $\si^\ell(X)=A_\ell X$. Thus
$R_0\otimes_{C_{R'}}C_i=\sum_{j=0}^{s-1}S_j$, where each $S_j$ is a
simple $\si^{\ell s}$-difference ring, and a domain. Because all
rings $R_j$ are isomorphic over $C_{R'}$, and all $S_j$ are
isomorphic over $C_{R'}$, $m(R'_i)$ is the product of $\ell s$ with
$m(S_0)=[D_{S_0^*}:C_{S_0^*}]$, where $S_0^*$ is the field of
fractions of $S_0$.  To show that $m(R'_i)=m(R')$, it therefore
suffices to show that $sm(S_0)=m(R_0)$. By Lemma~\ref{zpvlem2}.2,
$Fix(\si^{\ell
  s})(S_0^*)=Fix(\si^\ell)(R_0^*\otimes_{C_{R'}}C_i)=Fix(\si)(R'\otimes_{C_{R'}}C_i)=C_i$.

We  know that $D_{R_0^*}$ is a (cyclic) Galois extension of
$C_{R'}=Fix(\si^\ell)(R_0^*)$, and is therefore linearly disjoint
from $C_i$ over $D_{R^*_0}\cap C_i=C'_i$. Write
$C'_i=C_{R'}(\alpha)$, and let $a,b\in R_0$, $b\neq 0$, be such that
(inside $R_0^*$), $C_{R'}(a/b)=C'_i$. The minimal prime ideals of
$R_0\otimes_{C_{R'}}C_i$ are the ideals $Q_0,\ldots, Q_{r-1}$, where
$r=[C'_i:C_{R'}]$ and $Q_k$ is  generated by $\si^{k\ell}(a)\otimes
1 -\si^{k\ell}(b)\otimes \alpha$. This shows that $r=s$, since $s$
is also the number of minimal prime ideals of
$R_0\otimes_{C_{R'}}C_i$.

Let $e$ be a primitive idempotent of $R_0\otimes_{C_{R'}}C_i$ such
that
  $S_0=e(R_0\otimes_{C_{R'}}C_i)$. Then $eC_iD_{R_0^*}$ is a subfield of
  $S_0^*$, contained in $D_{S_0^*}$, and its degree over
  $eC_i=Fix(\si^{\ell s})(S_0^*)$ is the quotient of
  $[D_{R_0^*}:C_{R'}]$ by $[C'_i:C_{R'}]$, i.e., equals $m(R_0)/s$. To
  finish the proof, it therefore suffices to show that
  $D_{S_0^*}=eC_iD_{R_0}^*$.

Assume  that $c\in R^*_0\otimes_{C_{R'}}C_i$
  satisfies $\si^m(c)=c$ for some $m\neq 0$. Write $c=\sum_{k}a_k\otimes c_k$, where
  the $a_k$ are in $R^*_0$, and the $c_k$ are in $C_i$ and are linearly
  independent over $C_{R'}$. Then $\si^m(c)=c=\sum_k\si^m(a_k)c_k$,
  which implies $\si^m(a_k)=a_k$ for all $k$, and all $a_k$'s are in
  $D_{R_0^*}$. As every element of $D_{S_0^*}$ is of the form $ec$ for
  such a $c$ (Lemma \ref{zpvlem1}.3), this shows that
  $D_{S_0^*}=eC_iD_{R_0}^*$.  This finishes the proof that
  $m(R'_i)=m(R')$.

Consider now $R'\otimes_{KC}K_1$. It is the direct sum of $\ell s$
$\si^{\ell s}$-difference rings, each one being isomorphic to
$S_0\otimes_{KC}K_1$. Because $K_1$ is a regular extension of $KC$,
$S_0\otimes_{KC}K_1$ is a domain, of Krull dimension equal to
$\Krd(S_0)=\Krd(R')$. Inside its field of fractions (a $\si^{\ell
s}$-difference field) $K_1$ and $S^*_0$ are linearly disjoint over
$KC$, which implies that $C_{K_1}C_i$ is the field of constants of
$S_0\otimes _{KC}K_1$, $C_{K_1}D_{S_0^*}$ is the field of elements
fixed by some power of $\si$, and
$[C_{K_1}D_{S_0^*}:C_{K_1}C_i]=[D_{S_0}^*:C_i]=m(S_0)$. This shows
that $m(R'_i\otimes_{KC}K_1)=m(R')$ and finishes the proof.
\end{proof}
\begin{prop}\label{zp3}Assume that $C_K=D_K$.  Then all Picard-Vessiot rings for $\si(X)=AX$
  over $K$ have the same Krull dimension and the same $m$-invariant. \end{prop}
\begin{proof} Let $C$ be the algebraic closure of $C_K$, and
  let $R''$ be a Picard-Vessiot ring for $\si(X)=AX$ over $K$. By
  Proposition \ref{zp2}, $R'\otimes_{K}KC$ is the direct sum of
  finitely many Picard-Vessiot rings for $\si(X)=AX$ over $KC$, and each
  of these rings has the same Krull dimension and $m$-invariant as
  $R'$. The same statement holds for $R''$. On the other hand, by
  Proposition 1.9 of \cite{PuSi}, all Picard-Vessiot rings over $KC$ are
  isomorphic.
  \end{proof}

\begin{cor}\label{zcor1}Assume $D_K=C_K$. Let $R''=K[V,\det(V)^{-1}]$,
  where $\si(V)=AV$, and assume that $\Krd(R'')=\Krd(R')$ and that $R''$ has no nilpotent elements. Then
  $R''$ is a finite direct sum of Picard-Vessiot rings for
  $\si(X)=AX$.\end{cor}
\begin{proof} Because $R''$ has no
  nilpotent elements and is Noetherian, $(0)$ is the intersection of the finitely many
  prime minimal ideals of $R''$. Let $\calP$ be the set
  of  minimal prime ideals of $R''$. Then the intersection of any proper
  subset of $\calP$ is not $(0)$, i.e., no element of $\calP$ contains
  the intersection of the other elements of $\calP$. Also, if $P\in
  \calP$, then
  $\si(P)\in \calP$, and there exists $m>0$ such that $\si^m(P)=P$. Then
  $I_P=\bigcap_{i=0}^{m-1}\si^i(P)$ is a $\si$-ideal, which is proper if
  the orbit of $P$ under $\si$ is not all of $\calP$. Observe that for
  each $P\in\calP$, $\Krd(R''/P)\leq {\Krd}(R''/I_P)\leq
  \Krd(R'')={\Krd}(R')$, and that for some $P$ we have equality.

If $I$ is a maximal $\si$-ideal of $R''$, then
${\Krd}(R''/I)=\Krd(R')={\Krd}(R'')$ by Proposition \ref{zp2}, and
  this implies that $I$ is contained in some $P\in \calP$. Hence
  $I=I_P$ and $R''/I_P$ is a Picard-Vessiot ring. If $I=(0)$, then we are finished. Otherwise, $\calP$ contains
  some element $P_1$ not in the orbit of $P$ under $\si$. Observe that
  $I_{P_1}$ is contained in some maximal $\si$-ideal of $R''$, and is
  therefore maximal, by the same
  reasoning. Since the intersection of
  any proper subset of $\calP$ is non-trivial, $I_P+I_{P_1}$ is a
  $\si$-ideal of $R''$ which contains properly $I_P$, and therefore
  equals $1$. If $P_1, \ldots ,P_r$ are representatives from the
  $\sigma$-orbits in $\calP$, the Chinese Remainder Theorem then yields
  $R''\simeq \oplus_{i = 1}^r R''/I_{P_i}$. \end{proof}

\begin{prop}\label{zp4} Assume $C_K=D_K$. Then $KC_L[R]$ is a
  Picard-Vessiot ring for $\si(X)=AX$ over $KC_L$,
$${\Krd}(R')={\rm tr.deg}(L/KC_L),\quad\hbox{and
}[D_L:C_L]=m(R').$$\end{prop}
\begin{proof}Let us first assume that $R'$ is a domain. There is some
  generic difference field $\calu$ containing $R'$ and its field of
  fractions $R^*$, and which is
  sufficiently saturated. Because $L$ is a regular extension of $K$,
  there is some $K$-embedding $\varphi$ of $L$ into $\calu$, and we will
  denote by $T$ the image of $Y$ in $\calu$, and by $y$ the image of $Y$
  in $R'$. Then $\varphi(C_L)\subset C_\calu$, and there is some $B\in
  \GL_n(C_{\calu})$ such that $T=yB$. Hence
  $$KC_\calu[T,\det(T)^{-1}]=KC_\calu[y,\det(y)^{-1}].$$
By Proposition \ref{zp2}, $R'\otimes_KKC_\calu$ is a direct sum of
Picard-Vessiot rings of $\si(X)=AX$ over $KC_\calu$, and clearly one
of those is the domain $KC_\calu[y,\det(y)^{-1}]$. Thus
$\Krd(R')={\rm tr.deg}(R^*/K)={\rm tr.deg}(L/KC_L)$,
$D_{R^*}C_\calu=\varphi(D_L)C_\calu$, and $m(R')=[D_L:C_L]$.

This implies also that $K\varphi(C_L)[T,\det(T)^{-1}]$ is a simple
difference ring, and therefore a Picard-Vessiot ring for $\si(X)=AX$
over $K\varphi(C_L)$. Hence $KC_L[R]$ is a Picard-Vessiot extension
for $\si(X)=AX$ over $KC_L$.

In the general case, we replace $R'$ by $R_0$, $\si$ by $\si^\ell$,
find some generic sufficiently saturated $\si^\ell$-difference field
$\calu$ containing $R_0$, and a $K$-embedding $\varphi$ of the
$\si^\ell$-difference domain $L$ into $\calu$, and conclude as above
that $KFix(\si^\ell)[R_0]=KFix(\si^\ell)[\varphi(R)]$,
 that the Krull dimension of $R'$ equals ${\rm
  tr.deg}(L/KC_L)$, and that
$m(R_0)=[Fix(\si^\ell)(\varphi(D_L)):Fix(\si^\ell)]$.

Because $K$ and $D_L$ are linearly disjoint over $C_K$,
$[KD_L:KC_L]=[D_L:C_L]$, whence $D_{KC_L}=KC_L$, and by Corollary
\ref{zcor1}, the difference domain $KC_L[R]$ is a simple difference
ring, i.e., a Picard-Vessiot ring for $\si(X)=AX$ over $KC_L$. By
Proposition \ref{zp2}  $m(R')=[D_L:C_L]$.

We have $m(R')=\ell m(R_0)$, and $m(R_0)$ is the quotient of
$[D_L:C_L]$ by the greatest common divisor of $[D_L:C_L]$ and
$\ell$.
\end{proof}

\begin{cor}\label{zcor2}Assume that $C_K=D_K$.  Let
  $R''=K[V,\det(V)^{-1}]$ be a difference domain, where $\si(V)=AV$,  with field of
 fractions $L_1$, and assume that $C_{L_1}$ is a finite algebraic
 extension of $C_K$. Then $R''$ is a Picard-Vessiot
ring for $\si(X)=AX$ over $K$.\end{cor}

\begin{proof} Let $\calu$ be a sufficiently saturated generic difference field
containing $R''$, and let $\varphi$ be a $K$-embedding of $L$ into
$\calu$. Then $KC_\calu[\varphi(R)]=KC_\calu[R'']$. Hence
$\Krd(R'')=\Krd(R')$ and $R''$ is a Picard-Vessiot ring by Corollary
\ref{zcor1}. \end{proof}

\begin{cor}\label{zcor3}Assume that $C_K$ is algebraically closed. Then
  $\ell(R')=[D_L:C_L]$.\end{cor}
\begin{proof} Immediate from Proposition \ref{zp4} and the fact that
  $D_{R^*}=C_{R'}=C_K$.\end{proof}
\begin{cor}\label{zcor4}The difference ring $KC_L[R]$ is a
  Picard-Vessiot ring for $\si(X)=AX$ over $KC_L$. All
  Picard-Vessiot rings for $\si(X)=AX$ over $K$ have the same Krull
  dimension, which equals ${\rm tr.deg}(L/KC_L)$.\end{cor}
\begin{proof}Let $m=[D_K:C_K]$. Note that replacing $\si$ by some power
of $\si$ does not change the fields $D_K$ or $D_L$, and that
$Fix(\si^m)(K)=D_K$.  Therefore we can apply the previous results to
the equation $\sigma^m(X) = A_mX$ over $K$. By Corollary \ref{zcor2}
and because $KC_L[R]$ is a domain, $KC_L[R]$ is a Picard-Vessiot
ring for $\si^m(X)=A_mX$ over $KC_L$, and therefore a simple
$\si^m$-difference ring, whence a simple $\si$-difference ring, and
finally a Picard-Vessiot ring for $\si(X)=AX$ over $K$. \\[0.1in]
Let $R'=R/q$ be a Picard-Vessiot ring for $\si(X)=AX$ over $K$.
Assume
  first that $R'$ is a domain, and let $\calu$ be a generic difference
  field containing it. Because $L$ is a regular extension of $K$, there
is a $K$-embedding $\varphi$ of $L$
  into $\calu$, and from $KC_\calu[\varphi(R)]=KC_\calu[R']$ and Lemma
  \ref{zpvlem2}.1, we obtain the result. If $R'$ is not a domain, then
  we reason in the same fashion, replacing $R'$ by $R_0$ and $\si$ by
  $\si^\ell$, to obtain the result.\end{proof}
\begin{prop}\label{zp5}Assume that $C_{R'}=C_K=D_K$ and
$K\subset \calu$. Then
  $G_{R'}$ and $\HX$ are isomorphic.\end{prop}
\begin{proof} By Proposition \ref{p2}, we may replace $R'$ by
  $R'\otimes_KKD'_\calu$, and consider the ring
  $K\varphi(C_L)[T, \det(T)^{-1}]\otimes_{K\varphi(C_L)}KD'_\calu$, which is a Picard-Vessiot ring by
 Proposition~\ref{zp4} and  Corollary \ref{zcor1}.  We
  identify $1\otimes KD'_\calu$ with $KD'_\calu$. These two rings are
  isomorphic over $KD'_\calu$ by Proposition 1.9 of \cite{PuSi}, and it
  therefore suffices to show that
  $$\Aut(\varphi(L)\otimes_{K\varphi(C_L)}KD'_\calu/KD'_\calu)=\HX(D'_\calu).$$
 Inside
  $\varphi(L)\otimes_{K\varphi(C_L)}KD'_\calu$, $\varphi(L)\otimes 1$ and
  $KD'_\calu$ are linearly disjoint over $K\varphi(C_L)$. Hence, the
  algebraic loci of   $(T,\det(T)^{-1})$ over $K\varphi(C_L)$
  and over $KD'_\calu$ coincide. As $\HX$ was described as the subgroup
  of $\GL_n$ which leaves this algebraic set invariant, we get the result.
 \end{proof}
\subsection{Concluding remarks}\label{mt4}

\begin{remark}
{\bf Model-theoretic  Galois groups: definition and a bit of history}. 
{\rm Model-theoretic Galois groups first appeared in a paper by Zilber
  \cite{zilber} in the context of
$\aleph_1$-categorical theories, and under the name of {\em binding
    groups}. Grosso modo, the general situation is as 
follows: in a saturated model $M$ we have definable sets $D$ and $C$
such that, for some finite tuple $b$ in $M$, $D\subset \dcl(C,b)$ (one
then says that $D$ is $C$-internal). The 
group $\Aut(M/C)$ induces a group of (elementary) 
permutations of $D$, and it is this group which one calls the {\em Galois
group of $D$ over $C$}. In Zilber's context, this group and its action on $D$ are 
definable in $M$. One issue is therefore to find the correct
assumptions so that these Galois groups and their action are definable,
or at least, an intersection of definable groups. Hrushovski shows in his  PhD
thesis (\cite{HrPhD}) that  this is the case when the ambient theory is
stable.

Poizat, in \cite{poizat}, recognized the importance of elimination of imaginaries in
establishing the Galois correspondence for these Galois groups. He also
noticed that if $M$ is a differentially  
closed field of characteristic $0$ and $D$ is the set of solutions of
some linear differential equation over some differential subfield $K$ of
$M$, and  $C$  is the
field of constants of $M$, then the
model-theoretic Galois group coincides with the differential Galois
group introduced by Kolchin \cite{kolchin48}. This connection was further
explored by Pillay in a series of papers, see \cite{pillay}.  Note that
because the theory of differentially closed fields of characteristic $0$
eliminates quantifiers, this Galois group does coincide with the group
of $KC$-automorphisms of the differential field $KC(D)$.}\end{remark}

Since then, many authors studied or used Galois groups, under various
assumptions on the ambient theory, and in various contexts, either
purely model-theoretic (e.g., simple theories) or more algebraic (e.g. fields with Hasse
derivations).
In the context of generic difference fields, (model-theoretic) Galois
groups were investigated in  (5.11) of \cite{ChHr1} 
(a slight modification in 
the proof then gives the Galois group described in  section 4.1
of this paper). In positive characteristic $p$,   
the results generalize easily to  twisted difference equations of the form
$\si(X)=AX^{p^m}$, the field $Fix(\si)$ being then replaced
by $Fix(\tau)$, where $\tau:x\mapsto \si(x)^{p^{-m}}$.

\medskip
Recent work of Kamensky (\cite{kamensky}) isolates the common
ingredients underlying all the definability results on
  Galois groups, and in particular {\em very much weakens the assumptions}
  on the ambient theory (it is not even assumed to be complete).  
With the correct definition of
$C$-internality of the definable set $D$, he is able to 
show that a certain group of permutations of $D$ is definable in
$M$. These are just 
permutations, do not a priori preserve any relations of the language
other than equality. From this group, he is then able to show that
subgroups which preserve a (fixed) finite set of relations are also
definable, and that 
the complexity of the defining formula does not increase, or not too
much. For details, see section 3 of \cite{kamensky}. 

This approach of course applies to the set $D$ of solutions of a linear
system of difference equations (over a difference field $K$), and Kamensky
also obtains the result that $\Aut(KFix(\si)(D)/KFix(\si))$ is definable (see
section 5 in \cite{kamensky}). 

\begin{remark} {\rm A question arises in view of the proof of the
    general case of Proposition \ref{zp4}. When $R'$ is not a domain, we
    found an embedding of the $\si^\ell$-difference ring $R_0$ into a
    generic $\si^\ell$-difference field $\calu$. It may however happen
    that  $K$ is not relatively algebraically closed in $R_0^*$, even
    when $D_{R_0}=C_K$. Thus one can wonder:
    can one always find a generic difference field $\calu$ containing $K$, and
    such that there is a $K$-embedding of the $\si^\ell$-difference ring
    $R_0$ into $(\calu,\si^\ell)$? Or are there Picard-Vessiot rings for
    which this is impossible?}\end{remark}

\begin{remark}
{\bf Issues of definability}. {\rm It is fairly clear that the
algebraic group $\HX$ is defined over $\varphi(KC_L)$. On the other
hand, using the saturation of $\calu$ and the fact that $L$ is a
regular extension of $K$, we may choose another $K$-embedding
$\varphi_1$ of $L$ in $\calu$, and will obtain an algebraic group
$\HX_1$, which will be isomorphic to $\HX$ (via some matrix
$C\in\GL_n(C_\calu)$). It follows that $\HX$ is $K$-isomorphic to an
algebraic group $\HX_0$ defined over the intersections of all
possible $\varphi(KC_L)$, i.e., over $K$.}\end{remark} Observe that
the isomorphism between $\HX$ and $\HX_1$ yields an isomorphism
between $\HX(C_\calu)$ and $\HX_1(C_\calu)$, so that we will also
have an isomorphism between $\HX_0(C_\calu)$ and $\HX(C_\calu)$,
i.e., $\HX'$ is $K$-isomorphic to an algebraic subgroup of $\HX_0$
which is defined over $\overline{C_K}\cap C_\calu$. Thus when $C_K$
is algebraically closed,
it will be defined over $C_K$.\\[0.1in]
The Galois duality works as well for subgroups of $\HX(C_\calu)$
defined by equations (i.e., corresponding to algebraic subgroups of
$\HX'$, whose irreducible components are defined over $C_\calu$). It
works less well for arbitrary definable subgroups of $\HX(C_\calu)$.
In order for it to work, we need to replace $K(\calS)$ by its
definable closure $\dcl(K\calS)$, i.e., the subfield of $\calu$
which is fixed by all elements of $\Aut_{el}(\calu/K\calS)$. Because
the theory of $\calu$ eliminates imaginaries (1.10 in \cite{ChHr1}),
any orbit of an element of $\calS$ under the action of a definable
subgroup of $\HX(C_\calu)$ has a ``code'' inside $\dcl(K\calS)$.

\begin{remark}
{\bf Problems with the algebraic closure}. {\rm Assume that $\calu$
is a
  generic difference field containing $K$, and sufficiently
  saturated. Then if $K$ is not relatively algebraically closed in the
  field of fractions of $R_0$, we may not be able to find a
  $K$-embedding of $R_0$ into the $\si^\ell$-difference field $\calu$. Thus in particular, a priori not all
Picard-Vessiot domains $K$-embed into $\calu$. This problem of
course does not arise if we assume that $K$ is algebraically closed,
or, more
precisely, if we assume that\\
{\em All extensions of the automorphism $\si$ to the algebraic
closure of $K$ define $K$-isomorphic difference fields.}
}\end{remark} This is the case if $K$ has no finite (proper)
$\si$-stable extension, for instance when $K=\CX(t)$, with
$\si(t)=t+1$ and $\si$ the identity
on $\CX$. \\
However, in another classical case, this problem does arise: let
$q\in \CX$ be non-zero and not a root of unity, and let $K=\CX(t)$,
where $\si$ is the identity on $\CX$ and $\si(t)=qt$. Then $K$ has
non-trivial finite $\si$-stable extensions, and they are obtained by
adding $n$-th roots
of $t$.\\
Let us assume that, inside $\calu$, we have
$\si(\sqrt{t})=\sqrt{q}\sqrt{t}$. Let us consider the system
$$\si(Y)=-\sqrt{q}Y,\ \ Y\neq 0$$ over $K$. Then the Picard-Vessiot
ring is $R'=K(y)$, where $y^2=t$ and $\si(y)=-\sqrt{q}y$. Clearly
$R'$ does not embed in $\calu$. If instead we had considered this
system over $K(\sqrt{t})$, then the new Picard-Vessiot ring $R''$ is
not a domain anymore, because it will contain a non-zero solution of
$\si(X)+X=0$ (namely, $y/\sqrt{t}$). In both cases however the
Galois group is $\ZX/2\ZX$. And because $R'$ embeds in $R''$, it
also embeds
in $K(T)\otimes_{\varphi(C_L)}D'_\calu$. \\[0.1in]
This suggests that, when $C_K=D_K$,  if one takes $\calM$ to be the
subfield of $\calu$ generated over $KC_\calu$ by all tuples of
$\calu$ satisfying some linear difference equation over $K$, then
$\calM\otimes_{C_\calu}D'_\calu$ is a universal (full)
Picard-Vessiot ring of $KD'_\calu$. This ring is not so difficult to
describe in terms of $\calM$. Observe that $\calM$ contains
$D_\calu$. Thus $\calM\otimes_{C_\calu}D'_\calu$ is isomorphic to
$\calM\otimes_{D_\calu}(D_\calu\otimes_{C_\calu}D'_\calu)$. It is a
regular ring, with prime spectrum the Cantor space $\calC$ (i.e.,
the prime spectrum of $D_\calu\otimes_{C_\calu}D'_\calu$),  and
$\si$ acting on $\calC$. As a ring, it is isomorphic to the ring of
locally constant functions from $\calC$ to $\calM$.

It would be interesting to relate this ring to the universal
Picard-Vessiot rings defined in \cite{PuSi}.

\begin{remark}{\bf Saturation hypotheses}. {\rm The saturation
    hypothesis on $\calu$ is not really needed to define the
    model-theoretic  Galois
    group, since we only need $\calu$ to contain a copy of $L$ to define
    it. We also used it in the proof of Proposition \ref{zp4}, when we
    needed a $K$-embedding of $L$ into $\calu$. Thus, to define the
    model-theoretic Galois group, we only need $\calu$ to be a generic
    difference field containing $K$. Its field of constants will however
    usually be larger than $C_K$. Indeed, the
    field $C_\calu$ is always a pseudo-finite field (that is, a perfect,
    pseudo-algebraically closed field, with Galois group isomorphic to
    $\hat\ZX$). However, one can show that if $F$ is a pseudo-finite
    field of characteristic $0$, then there is a generic difference
    field $\calu$ containing $F$ and such that $C_\calu=F$. Thus, the
    field of constants of $\calu$ does not need to be much larger than
    $C_K$.  In the general case, a general non-sense construction allows
    one to find a
    pseudo-finite field $F$ containing $C_K$ and of transcendence degree
    at most $1$ over $C_K$.}\end{remark}

\begin{remark}{\bf A partial description of the maximal $\si^\ell$-ideal $p$ of
    $R$}.
{\rm We keep the notation of the previous subsections, and will
first assume that the Picard-Vessiot ring $R'=R/q$ is a domain
contained in $\calu$.}\end{remark} We will describe some of the
elements of $q$. Write
 $C_L=C_K(\alpha_1,\ldots,\alpha_m)$, and
$\alpha_i=f_i(Y)/g_i(Y)$, where $f_i(Y), g_i(Y)\in K[Y]$ are
relatively prime. Then $\si(f_i)(AY)$ and $\si(g_i)(AY)$ are also
relatively prime. Looking at the divisors defined by these
polynomials, we obtain that there is some $k_i\in K$ such that
$\si(f_i)(AY)=k_if_i(Y)$ and $\si(g_i)(AY)=k_ig_i(Y)$. Then
$(q,f_i(Y))$ and $(q,g_i(Y))$ are $\si$-ideals. By the maximality of
$q$, this implies that either $f_i(Y)$ and $g_i(Y)$ are both in $q$, or
else, say if $f_i(Y)\notin q$, that there is some $c_i\in C_{R'}$
such that $g_i(y)=c_if_i(y)$, because $f_i(y)$ is invertible in
$R'$. If $P_i(Z)$ is the minimal monic polynomial of $c_i$ over $C_K$ and
is of degree $r$, then $g_i(Y)^rP_i(g_i(Y)/f_i(Y))\in q$. In case
$C_{R'}=C_K$ (this is the case for instance if
$C_K$ is algebraically closed), then $c_i\in C_K$, and 
  $g_i(Y)-c_if_i(Y)$ will belong to $q$. (Note also
that if $k_i=k_j$, then also for some $d_j\in C_K$ we will have
$f_j(Y)-d_jf_i(Y)\in q$, and $g_j(Y)-c_jd_jf_i(Y)\in q$). The
$\si$-ideal $I$ generated by all these polynomials in $R$ could
all of $q$. In any case one shows easily that $q$ is a minimal prime
ideal containing it (because $KC_L[Y,\det(Y)^{-1}]$ and $R/I$ have
the same
Krull dimension, which is also the Krull dimension of $R'$).

A  better result is obtained by Kamensky in \cite{kamensky} Proposition
33: if $C_{R'}=C_K$, and instead of 
looking at a generating set of $C_L$ over $C_K$ one applies the same
procedure to all elements of $C_L$, one obtains a generating set of the
ideal $q$.

In case $R'$ is not a domain, we reason in the same fashion to get a
partial description of the $\si^\ell$-ideal $p$.

\bibliographystyle{plain}
\newcommand{\SortNoop}[1]{}

\bigskip\noindent{\bf Addresses}

\smallskip\noindent
UFR de Math{\'e}matiques\\
Universit{\'e} Paris 7 - Case 7012\\ 
2 place Jussieu\\
75251 Paris cedex 05\\
France. \\
{\tt zoe@logique.jussieu.fr}

\smallskip\noindent
Institut de
Math{\'e}matiques\\
Universit{\'e} Paris 6\\
Tour 46 5e {\'e}tage\\
Boite 247\\
4 place Jussieu\\
75252 Paris cedex 05\\
France\\
{\tt hardouin@math.jussieu.fr}

\smallskip\noindent
North
Carolina State University\\
Department of Mathematics\\
Box 8205\\
Raleigh, North Carolina 27695-8205\\
USA\\
{\tt singer@math.ncsu.edu}.

\end{document}